\documentclass[reqno]{amsart}
\usepackage{amsfonts,amsmath,amsthm,amssymb,mathrsfs}
\usepackage{color}
\usepackage{amsmath,amssymb,amsfonts,bm}
\usepackage{graphicx}
\usepackage{CJKutf8}
\usepackage{xcolor}
\numberwithin{equation}{section}
\usepackage[pdfstartview=FitH,colorlinks=true]{hyperref}
\hypersetup{urlcolor=red, linkcolor=blue,citecolor=red, CJKbookmarks=true}
\allowdisplaybreaks

\newtheorem{theorem}{Theorem}[section]

\newtheorem{lemma}[theorem]{Lemma}

\theoremstyle{definition}

\newtheorem{remark}{Remark}[section]

\newcommand{\abs}[1]{\left\vert#1\right\vert}

\newcommand{\norm}[1]{\left\Vert#1\right\Vert}

\newcommand{\inner}[1]{\left(#1\right)}

\makeatletter

\newcommand{\Rmnum}[1]{\expandafter\@slowromancap\romannumeral #1@}
\makeatother

\DeclareMathOperator\sgn{sgn}
\DeclareMathOperator\divg{div}
\DeclareMathOperator\curl{curl}

\title[Gevrey class regularity of solutions to the magnetohydrodynamics equations]
{On the Gevrey regularity of solutions\\
to the 3D ideal MHD equations}
\author{Feng Cheng and Chao-Jiang Xu}
\date{}

\address{\noindent \textsc{Feng Cheng, School of Mathematics and Statistics, Wuhan university 430072, Wuhan, P.R. China}}
\email{chengfengwhu@whu.edu.cn}

\address{\noindent \textsc{Chao-Jiang Xu,
School of Mathematics and Statistics, Wuhan university 430072, Wuhan, P.R. China\\
and\\
Universit\'e de Rouen, CNRS UMR 6085, Laboratoire de Math\'ematiques, 76801 Saint-Etienne du Rouvray, France}}
\email{Chao-Jiang.Xu@univ-rouen.fr}
\begin{document}


\begin{abstract}
In this paper, similar to the incompressible Euler equation, we prove the propagation of the Gevrey regularity of solutions to the three-dimensional incompressible ideal magnetohydrodynamics (MHD) equations. We also obtain
an uniform estimate of Gevery radius for the solution of MHD equation.
\end{abstract}

\maketitle

\section{Introduction}
The three-dimensional (3D) incompressible ideal MHD equations on the torus $\mathbb{T}^3$ take the form,
\begin{equation}\label{1.1}
\left\{
\begin{aligned}
 &\frac{\partial u}{\partial t}+u\cdot\nabla u-h\cdot\nabla h+\nabla (p+\frac{1}{2}\abs{h}^2)=0,\ &x\in\mathbb{T}^3,\ t>0,\\
 &\frac{\partial h}{\partial t}+u\cdot\nabla h-h\cdot\nabla u=0,\ &x\in\mathbb{T}^3,\ t>0,\\
 &\nabla\cdot u=0,\quad \nabla\cdot h=0,\ &x\in\mathbb{T}^3,\ t\geq0,\\
 &u(x,0)=u_0(x),\ h(x,0)=h_0(x),\ &x\in\mathbb{T}^3,
\end{aligned}
\right.
\end{equation}
where $u(x,t)=(u_1, u_2, u_3)(x,t), h(x,t)=(h_1,h_2,h_3)(x,t)$, represent fluid velocity field, magnetic field at point $x=(x_1, x_2, x_3)\in\mathbb{T}^3$ at time $t$, and $p=p(x,t)$ represents the scalar pressure.
Note that the incompressibility $\nabla\cdot h=0$ needs only be required at $t=0$, and it then holds for all $t>0$. As for the classical Euler equation, we transform the equations \eqref{1.1} to the following form after taking curl operator on both sides,
\begin{equation}\label{1.2}
\left\{
\begin{aligned}
 &\frac{\partial \omega}{\partial t}+u\cdot\nabla \omega-h\cdot\nabla J=\omega\cdot\nabla u-J\cdot\nabla h, \\
 &\frac{\partial J}{\partial t}+u\cdot\nabla J-h\cdot\nabla \omega=\omega\cdot\nabla h-J\cdot\nabla u, \\
 &u=\mathcal{K}* \omega,\quad h=\mathcal{K}* J, \\
 &\omega|_{t=0}=\omega_0=\curl u_0,\quad J|_{t=0}=J_0=\curl h_0,
\end{aligned}
\right.
\end{equation}
where $\mathcal{K}$ is the three dimensional Biot-Savart kernel, $\omega=\nabla\times u$ and $J=\nabla\times h$ denote the vorticity and current density, see \cite{Wu}.

In magneto-fluid mechanics magnetohydrodynamics equations (MHD) describes the dynamics of electrically conducting fluids arising from plasmas, liquid metals, and salt water or electrolytes, see \cite{LL, ST}.
There is no global well-posedness for the incompressible MHD equations \eqref{1.1} in general case except for small pertubation near the trivial steady solution (see, for instance \cite{LLP} and \cite{YZ}). The local existence and uniqueness of ${H}^r$-solution, for $r>5/2$, of the Cauchy problem \eqref{1.2} was proved in \cite{SP} following the method of Temam \cite{T} and Kato and Lai \cite{KL}. Caflisch, Klapper and Steele \cite{CKS} extended the well-known Beal, Kato and Majda criterion \cite{BK} for incompressible Euler equations to the cases of incompressible ideal MHD equations. Precisely, they proved that if the maximal time of existence $T$ is finite, then
\begin{equation}\label{1.3}
\int_{0}^{T}\big(\norm{\omega(\cdot,t)}_{L^\infty}+\norm{J(\cdot,t)}_{L^\infty}\big)dt=\infty.
\end{equation}
For more work about the blow-up criterion, please refer to \cite{CCM,ZX} and reference therein.
In this paper we study the Gevrey class regularity of the ${H}^r$-solutions to equations \eqref{1.2} on the torus $\mathbb{T}^3$ using the Fourier space method introduced by Foias and Temam \cite{FT}. In that paper, the authors studied the Gevrey class regularity of Navier-Stokes equations and proved that the solutions are analytic in time with values in Gevrey class for initail data only in Sobolev space $H^1$ with divergence free.
Levermore and Oliver \cite{LO} applied this method to study the propagation of analyticity of the solutions to the so-called lake and great lake equations.
Later, Kukavica and Vicol \cite{KV1} improved the results of Levermore and Oliver by showing that the radius of space analyticity decays algebraically on $\exp{\int_0^t \|\nabla u^E(\cdot,s)\|_{L^\infty}ds}$, where $u^E$ is the solution of incompressible Euler equations.  The purpose of this paper is to generalize the results of Kukavica and Vicol to 3D incompressible ideal MHD equations.

When considering viscous and resistive incompressible MHD equations, Kim \cite{SK} had investigated the Gevrey class regularity of the strong solutions and proved a parallel result as Foias and Temam \cite{FT} on Navier-Stokes equations. For regularized MHD equations, Yu and Li \cite{YL} studied Gevrey class regularity of the strong solutions to the MHD-Leray-alpha equations and Zhao and Li \cite{ZL} studied analyticity of the global attractor of the so-called MHD-Voight equations following the method of \cite{KLT}. In the whole space $\mathbb{R}^3$, Wang and Li \cite{WL} studied the global existence of solutions to the viscous and resistive MHD equations in the so-called Lei-Li-Gevrey space and Weng \cite{W} studied the analyticity of solutions to the Hall-MHD equations.  However, these aforementioned works are mainly concerned the viscous and resistive MHD equations (or regularized MHD equations). We see no results of Gevrey class regularity for the ideal MHD equations yet by far, and this is the motivation of our work.

The paper is organized as follows. In Section \ref{Section 2}, we will give some notations and state our main results. In Section \ref{Section 3}, we first recall some known results and then give some lemmas which are needed to prove the main Theorem. In Section \ref{Section 4}, we finish the proof of Theorem \ref{Theorem 2.1}.

\section{Notations and Main Theorem}\label{Section 2}

In this section we will give some notations and function spaces which will be used throughout the following arguments. Throughout the paper, $C$ denotes a generic constant which may vary from line to line.

Let $r\geq 0$ be a constant. Denote by ${H}^r(\mathbb{T}^3)$ the mean zero vector function space of fractional Sobolev space,
\begin{equation*}
\begin{split}
  {H}^r(\mathbb{T}^3)=\bigg\{v(x)=\sum_{k\in\mathbb{Z}^3}\hat{v}_k e^{ik\cdot x} &: \hat{v}_0=0,\ \overline{\hat{v}_k}=\hat{v}_{-k}, \\
   & \norm{v}_{{H}^r}=(2\pi)^3 \sum_{k\in\mathbb{Z}^3}(1+\abs{k}^2)^{r}\abs{\hat{v}_k}^2<\infty\bigg\},
\end{split}
\end{equation*}
where $\hat{v}_k$ is the $k$-th vector Fourier coefficient defined by
$$
 \hat{v}_k=\int_{\mathbb{T}^3}v(x)e^{-i k\cdot x}dx,\quad i=\sqrt{-1}.
$$
The operator $\Lambda$ is defined as follows
$$
\Lambda v(x):=\sum_{k\in\mathbb{Z}^3\backslash\{0\}}\abs{k}_1 \hat{v}_k e^{ik\cdot x},
$$
here $v\in H^1(\mathbb{T}^3)$ we used the notation $\abs{k}_1=\abs{k_1}+\abs{k_2}+\abs{k_3}$. Let $m=1,2,3$, define $\Lambda_m$ and $H_m$ as follows,
$$
\Lambda_m v(x):=\sum_{k\in\mathbb{Z}^3\backslash\{0\},k_m\neq0}\abs{k_m} \hat{v}_k e^{ik\cdot x},\quad H_m v(x):=\sum_{k\in\mathbb{Z}^3\backslash\{0\},k_m\neq0}\sgn(k_m)\hat{v}_k e^{ik\cdot x},
$$
for all $v\in {H}^1(\mathbb{T}^3)$.

Let $s\geq1$ be a real number.
For any multi-index $\alpha=(\alpha_1,\alpha_2,\alpha_3)$ in $\mathbb{N}^3$, we denote $\abs{\alpha}=\alpha_1+\alpha_2+\alpha_3$.
Usually, we say that a smooth function $f(x)\in C^\infty(\mathbb{R}^3)$ is uniformly of Gevrey class s, if there exists $C,\tau>0$ such that
\begin{equation}\label{2.1}
 \abs{\partial^\alpha f(x)}\leq C\frac{\abs{\alpha}!^s}{\tau^{\abs{\alpha}}},
\end{equation}
for all $x\in\mathbb{R}^3$ and all multi-index $\alpha\in\mathbb{N}^3$. When $s=1$, $f$ is real analytic. The constant $\tau$ in \eqref{2.1} is called the radius of Gevrey class regularity. Inspired by Foias and Temam \cite{FT}, the Gevrey space on the torus can be characterized by the decay of the Fourier coefficients, see for instance \cite{KV1,LO}.

In this paper we inherit the notations of the function space of Gevrey class $s$ used in \cite{KV1}. For fixed $r,\tau\geq0$ and $m=1,2,3$, let
\begin{equation*}
  \mathscr{D}(\Lambda_m^r e^{\tau\Lambda_m^{1/s}})=\bigg\{v\in H^r(\mathbb{T}^3);\ \divg v=0,\ \norm{\Lambda_m^r e^{\tau\Lambda_m^{1/s}}v}_{L^2}<\infty\bigg\},
\end{equation*}
where
\begin{equation*}
  \norm{\Lambda_m^r e^{\tau\Lambda_m^{1/s}}v}_{L^2}=(2\pi)^3\sum_{k\in\mathbb{Z}^3}\abs{k_m}^{2r}e^{2\tau\abs{k_m}^{1/s}}\abs{\hat{v}_k}^2.
\end{equation*}
For $\tau,r\geq0$, set
$$
 X_{r,\tau,s}=\bigcap_{m=1}^3 \mathscr{D}(\Lambda_m^r e^{\tau \Lambda_m^{1/s}}),\quad \norm{v}_{X_{r,\tau,s}}^2=\sum_{m=1}^{3}\norm{\Lambda_m^r e^{\tau\Lambda_m^{1/s}}v}_{L^2}^2,
$$
and
$$
Y_{r,\tau,s}=X_{r+\frac{1}{2s},\tau,s}.
$$
The function spaces defined above are showed to be equivalent with the usual definition of Gevrey class $s$ and we still call the parameter $\tau$ the radius of Gevrey class s, see \cite{KV1, LO,SK} for detailed description.

With these notations, we can state our main results.

\begin{theorem}\label{Theorem 2.1}
Let $r> \frac{5}{2}+\frac{3}{2s}, s\geq1$ be fixed constants. If $(u_0, h_0)$ are divergence-free and $(\omega_0, J_0)=(\curl u_0,\curl h_0)\in X_{r,\tau_0,s}$ with $\tau_0>0$. Then the equation \eqref{1.2} admits a unique solution $(\omega, J)\in L^\infty([0, T[;H^r(\mathbb{T}^3) )$ such that,
\begin{equation*}
(\omega,J)\in L^\infty\big([0,T[,X_{r,\tau(\cdot),s}\big), % \bigcap L^2\big([0,T),Y_{r,\tau(\cdot),s}\big),
\end{equation*}
where $0<T$ is the life-span of ${H}^r$-solution $(u,h)$ to equations \eqref{1.1}. Moreover the Gevery radius $\tau(t)$ is a decreasing function of $t$ with $\tau(0)=\tau_0$ and satisfies, for $0\leq t<T$,
\begin{equation*}%\label{2.3}
\tau(t)\geq \exp\bigg(-C\int_{0}^{t}(\norm{\nabla u(\cdot,\sigma)}_{L^\infty}+\norm{\nabla h(\cdot,\sigma)}_{L^\infty}d\sigma\bigg)\big(\tau_0^{-1}+C_0 t+\frac{C_1}{2}t^2\big)^{-1},
\end{equation*}
where $C>0$ is a constant depending only on $r, s$, while $C_0$
and $C_1$ have additional dependence on the initial data.
\end{theorem}
\begin{remark}
In the case $s=1$ and $h=constant$, Theorem \ref{Theorem 2.1} recovers the result of Kukavica and Vicol \cite{KV1} for incompressible Euler equation. And we remarked that in the case $s=1$, we need only $r>\frac{7}{2}$ in Theorem \ref{Theorem 2.1}.
\end{remark}

\begin{remark}
The smooth solution criterion \eqref{1.3} in \cite{CKS} states that the solution remain smooth to $T$ as long as $\int_{0}^{T}(\norm{\omega(\cdot,t)}_{L^\infty}+\norm{J(\cdot,t)}_{L^\infty})dt<\infty$.
\end{remark}

\section{The estimate of the nonlinear terms}\label{Section 3}
In order to prove the main Theorem \ref{Theorem 2.1}, we recall the following results about the local existence and uniqueness of ${H}^r$-solution of the ideal MHD equations \eqref{1.1},
\begin{theorem}[Caflisch-Klapper-Steele, \cite{CKS}]\label{Theorem 3.2}
Let $r\geq3$. If $u_0, h_0\in {H}^{r}(\mathbb{T}^3)$ are divergence-free. Then equations \eqref{1.1} admit a unique solution $(u,h)$ such that
  \begin{equation*}%\label{3.1}
    (u,h) \in C\big([0,T[, {H}^{r}(\mathbb{T}^3)\big)\bigcap C^1\big([0,T[,{H}^{r-1}(\mathbb{T}^3)\big)
  \end{equation*}
where $0<T<\infty$ is the maximal existence time of ${H}^{r}$-solution, namely $T$ stasifies
\begin{equation*}
  \int_{0}^{T}\norm{\omega(\cdot,t)}_{L^\infty(\mathbb{T}^3
  )}+\norm{J(\cdot,t)}_{L^\infty(\mathbb{T}^3)}dt=\infty.
\end{equation*}
\end{theorem}
The proof of Theorem \ref{Theorem 3.2} can be found in \cite{CKS}, which is analogue of the Beal-Kato-Majda Theorem on the Euler equations. With this Theorem and the Biot-Savart law, one can easily deduce the existence of solution $(\omega, J)\in C([0,T);H^r(\mathbb{T}^3))$ to equations \eqref{1.2} if the initial data $(\omega_0,J_0)=(\curl u_0,\curl h_0)\in H^r(\mathbb{T}^3)$.

In the following we state some Lemmas concerning the estimates of the nonlinear terms in equation. First, we recall two useful Lemmas from \cite{KV1}.
\begin{lemma}[Lemma 3.1 of \cite{KV1}]\label{lemma 3.2}
Let $w\in X_{r,\tau,s}$, for $\tau\geq0$ and $r\geq1$. Then for $m=1,2,3$ we have
\begin{equation*}
\|\Lambda_m^r w\|_{L^2}\leq \|\Lambda\Lambda_m^{r-1}w\|_{L^2}\leq C\|w\|_{H^r}
\end{equation*}
and
\begin{equation*}
\|\nabla H_m\Lambda_m^{r-1}e^{\tau\Lambda_m^{1/s}}w\|_{L^2}\leq \|\Lambda \Lambda_m^{r-1}e^{\tau\Lambda_m^{1/s}}w\|_{L^2}\leq C\|w\|_{X_{r,\tau,s}},
\end{equation*}
where $C$ is a positive constant.
\end{lemma}
And we recall the Biot-Savart law in \cite{MB}.
\begin{lemma}\label{Lemma 3.3}
Let $w\in X_{r,\tau,s}$, for $\tau\geq0$ and $r\geq1$. Let $v=\mathcal{K}*w$. Then for $m=1,2,3$ we have
  \begin{equation*}
    \norm{\Lambda_m^{r+1} v}_{L^2}\leq \norm{\Lambda\Lambda_m^r v}_{L^2}\leq C\norm{w}_{H^r}
  \end{equation*}
and
  \begin{equation*}
    \norm{\Lambda_m^{r+1}e^{\tau\Lambda_m^{1/s}} v}_{L^2}\leq \norm{\Lambda\Lambda_m^r e^{\tau\Lambda_m^{1/s}} v}_{L^2}\leq C\norm{w}_{X_{r,\tau,s}},
  \end{equation*}  
where $C$ is a positive constant independent of $v,w$.
\end{lemma}
The proof is standard by Calder\'on Zygmund theory, we thus omit the proof. In order to estimate the nonlinear terms in equations \eqref{1.2}, we first recall the Lemma 2.5 in \cite{KV1}, in which the authors proved the case of $s=1$. Denote the $L^2$-norm and and the inner product by $\norm{\,\cdot\,}_{L^2(\mathbb{T}^3)}$ and $\inner{\cdot,\cdot}_{L^2(\mathbb{T}^3)}$ respectively.
\begin{lemma}[Lemma 2.5 of \cite{KV1}]\label{Lemma 3.4}
Let $m=1,2,3$ and $\omega \in Y_{r,\tau,s}$, where $r>\frac{5}{2}+\frac{3}{2s}$. If $u=\mathcal{K}* \omega$, where $\mathcal{K}$ is the Biot-Savart kernel, then
  \begin{align}\label{3.1}
    &\abs{\inner{u\cdot\nabla \omega,\Lambda_m^{2r}e^{2\tau\Lambda_m^{1/s}}\omega}_{L^2(\mathbb{T}^3)}}+\abs{\inner{\omega\cdot\nabla u,\Lambda_m^{2r}e^{2\tau\Lambda_m^{1/s}}\omega}_{L^2(\mathbb{T}^3)}} \notag\\
    &\leq C\big(\tau\norm{\nabla u}_{L^\infty}+\tau^2\norm{\omega}_{H^r}+\tau^2\norm{\omega}_{X_{r,\tau,s}}\big)\norm{\omega}_{Y_{r,\tau,s}}^2 \notag\\
    &\quad+C\big(\norm{\nabla u}_{L^\infty}\norm{\omega}_{X_{r,\tau,s}}+(1+\tau)\norm{\omega}_{H^r}^2\big)\norm{\omega}_{X_{r,\tau,s}},
  \end{align}
 where the positive constant $C$ depends on $r$ and $s$.
\end{lemma}
We remark that for $s>1$ there are some minor changes in the proof which cause the condition $r>\frac{5}{2}+\frac{3}{2s}$, and we show the details in the proof of the following Lemma.
First we introduce the following notation
\begin{equation*}
  \Psi=(\omega,J),
\end{equation*}
and the corresponding norm
$$
 \norm{\Psi}_{H^r}^2=\norm{\omega}_{H^r}^2+\norm{J}_{H^r}^2,\ \norm{\Psi}_{X_{r,\tau,s}}^2=\norm{\omega}_{X_{r,\tau,s}}^2+\norm{J}_{X_{r,\tau,s}}^2.
$$
With very similar method as Lemma 2.5 of \cite{KV1}, we can obtain the following Lemma.
\begin{lemma}\label{lemma 3.5}
Let $m=1,2,3$ and $\omega, J \in Y_{r,\tau,s}$, where $r>\frac{5}{2}+\frac{3}{2s}$. If $u=\mathcal{K}* \omega, h=\mathcal{K}* J$, where $\mathcal{K}$ is the Biot-Savart kernel, then
 \begin{align}\label{3.2}
    &\abs{\inner{u\cdot\nabla J,\Lambda_m^{2r}e^{2\tau\Lambda_m^{1/s}}J}_{L^2(\mathbb{T}^3)}}+\abs{\inner{J\cdot\nabla u,\Lambda_m^{2r}e^{2\tau\Lambda_m^{1/s}}J}_{L^2(\mathbb{T}^3)}} \notag\\
    &\leq C(\tau\norm{\nabla u}_{L^\infty}+\tau^2\norm{\Psi}_{{H}^r}+\tau^2\norm{\Psi}_{X_{r,\tau,s}})\norm{\Psi}_{Y_{r,\tau,s}}\norm{J}_{Y_{r,\tau,s}} \notag\\
    &\quad+C\big[\norm{\nabla u}_{L^\infty}\norm{J}_{X_{r,\tau,s}}+\norm{\nabla h}_{L^\infty}\norm{\omega}_{X_{r,\tau,s}}+(1+\tau)\norm{\Psi}_{{H}^r}^2\big] \norm{J}_{X_{r,\tau,s}},
  \end{align}
where $C$ is a positive constant.  
\end{lemma}

\begin{proof}
Let $m\in\{1,2,3\}$. In order to estimate $\abs{\inner{u\cdot\nabla J,\Lambda_m^{2r}e^{2\tau\Lambda_m^{1/s}}J}_{L^2(\mathbb{T}^3)}}$, we appeal to the cancellation property $\inner{u\cdot\nabla \Lambda_m^r e^{\tau \Lambda_m^{1/s}}J,\Lambda_m^r e^{\tau \Lambda_m^{1/s}}J}_{L^2(\mathbb{T}^3)}=0$ with notification $\divg u=0$. Using Plancherel's theorem we obtain
\begin{align}\label{3.3}
   &\inner{u\cdot\nabla J,\Lambda_m^{2r}e^{2\tau\Lambda_m^{1/s}}J}_{L^2(\mathbb{T}^3)} \notag\\
   &=\inner{u\cdot\nabla J,\Lambda_m^{2r}e^{2\tau\Lambda_m^{1/s}}J}_{L^2(\mathbb{T}^3)}
   -\inner{u\cdot\nabla \Lambda_m^r e^{\tau \Lambda_m^{1/s}}J,\Lambda_m^r e^{\tau \Lambda_m^{1/s}}J}_{L^2(\mathbb{T}^3)} \notag \\
   &= i(2\pi)^3 \sum_{j+k+\ell=0}(|\ell_m|^r e^{\tau |\ell_m|^{1/s}}-|k_m|^r e^{\tau |k_m|^{1/s}})(\hat{u}_j\cdot k)(\hat{J}_k\cdot\hat{J}_\ell)|\ell_m|^r e^{\tau |\ell_m|^{1/s}} \notag \\
   &=  i(2\pi)^3 \sum_{j+k+\ell=0}(|\ell_m|^r -|k_m|^r) e^{\tau |k_m|^{1/s}}(\hat{u}_j\cdot k)(\hat{J}_k\cdot\hat{J}_\ell)|\ell_m|^r e^{\tau |\ell_m|^{1/s}}  \notag\\
   &\quad+ i(2\pi)^3 \sum_{j+k+\ell=0}|\ell_m|^r( e^{\tau |\ell_m|^{1/s}}- e^{\tau |k_m|^{1/s}})(\hat{u}_j\cdot k)(\hat{J}_k\cdot\hat{J}_\ell)|\ell_m|^r e^{\tau |\ell_m|^{1/s}} \notag\\
   &:=T_{u,J,J}^{(1)}+T_{u,J,J}^{(2)},
\end{align}
with $j,k,\ell\in\mathbb{Z}^3$. Recall that $\hat{u}_0=\hat{J}_0=0$. In order to estimate ${T}^{(1)}_{u,J,J}$, we first expand $\abs{\ell_m}^r-\abs{k_m}^r$ by means of mean value theorem,
\begin{align}\label{3.4}
  \abs{\ell_m}^r-\abs{k_m}^r &=r(\abs{\ell_m}-\abs{k_m})(\theta_{m,k,\ell}\abs{\ell_m}+(1-\theta_{m,k,\ell})\abs{k_m})^{r-1} \notag\\
  &=r(\abs{\ell_m}-\abs{k_m})\big[\big(\theta_{m,k,\ell}\abs{\ell_m}+
  (1-\theta_{m,k,\ell})\abs{k_m}\big)^{r-1}-\abs{k_m}^{r-1}\big] \notag\\
  &\quad+r(\abs{\ell_m}-\abs{k_m})\abs{k_m}^{r-1},
\end{align}
where $\theta_{m,k,\ell}\in(0,1)$ is a constant. Since $j+k+\ell=0$, we have, by the triangle inequality,
\begin{equation*}%\label{3.11}
\begin{aligned}
  &\quad\abs{r(\abs{\ell_m}-\abs{k_m})\big[(\theta_{m,k,\ell}\abs{\ell_m}+(1-\theta_{m,k,\ell})\abs{k_m})^{r-1}-\abs{k_m}^{r-1}\big]}\\
  &\leq C\abs{j_m}^2(\abs{j_m}^{r-2}+\abs{k_m}^{r-2}).
\end{aligned}
\end{equation*}
Since $j_m+k_m+\ell_m=0$, we have the following decomposition, introduced by \cite{KV1},
\begin{align}\label{3.5}
  \abs{\ell_m}-\abs{k_m} &=\abs{j_m+k_m}-\abs{k_m}\notag\\
   &=j_m\sgn(k_m)+2(j_m+k_m)\sgn(j_m)\chi_{\{\sgn(k_m+j_m)\sgn(k_m)=-1\}}.
\end{align}
In the region $\{\sgn(k_m+j_m)\sgn(k_m)=-1\}$, we have $\abs{k_m}\leq\abs{j_m}$. Then with use of $e^\xi\leq e+\xi^2 e^{\xi}$ for $\xi=\tau|k_m|^{1/s}\geq0$, $|\hat{u}_j\cdot k|\leq C|\hat{u}_j||k|_1$ and Plancherel's theorem we have, by discrete Cauchy-Schwartz inequality,
\begin{align}\label{3.6}
  &\abs{{T}^{(1)}_{u,J,J}} \notag\\
  &\leq  C\sum_{j+k+\ell=0}\bigg\{(\abs{j_m}^{r}+\abs{j_m}^2\abs{k_m}^{r-2})(e+\tau^2\abs{k_m}^{2/s}e^{\tau\abs{k_m}^{1/s}})
  |\hat{u}_j|\abs{k}_1|\hat{J}_k||\hat{J}_\ell| \notag\\
  &\quad\quad\quad\quad\quad\quad\quad\quad\times\abs{\ell_m}^r e^{\tau\abs{\ell_m}^{1/s}}\bigg\} \notag\\
  &\quad+C\bigg|\sum_{j+k+\ell=0}j_m\sgn(k_m)\abs{k_m}^{r-1}e^{\tau\abs{k_m}^{1/s}}(\hat{u}_j\cdot k)(\hat{J}_k\cdot\hat{J}_\ell)\abs{\ell_m}^r e^{\tau\abs{\ell_m}^{1/s}}\bigg| \notag\\
  &\leq C\norm{\nabla u}_{L^\infty}\norm{J}_{X_{r,\tau,s}}\norm{\Lambda_m^r e^{\tau\Lambda_m^{1/s}}J}_{L^2}+C\norm{\omega}_{{H}^r}\norm{J}_{{H}^r}\norm{\Lambda_m^r e^{\tau\Lambda_m^{1/s}}J}_{L^2} \notag\\
  &\quad+C\tau^2\norm{\omega}_{{H}^r}\norm{J}_{Y_{r,\tau,s}}\norm{\Lambda_m^{r+\frac{1}{2s}} e^{\tau\Lambda_m^{1/s}}J}_{L^2},
\end{align}
where $C$ is some constant depending on $r$. The presence of the supremum of the velocity gradient, the innovative point of \cite{KV1}, is due to the use of Plancherel's theorem in the following form,
\begin{align*}%\label{3.14}
  &\quad\bigg|\sum_{j+k+\ell=0}j_m\sgn(k_m)\abs{k_m}^{r-1}e^{\tau\abs{k_m}^{1/s}}(\hat{h}_j\cdot k)(\hat{\omega}_k\cdot\hat{J}_\ell)\abs{\ell_m}^r e^{\tau\abs{\ell_m}^{1/s}}\bigg| \\
  &=\bigg|\inner{\partial_m h\cdot\nabla H_m\Lambda_m^{r-1}e^{\tau\Lambda_m^{1/s}}\omega,\Lambda_m^r e^{\tau \Lambda_m^{1/s}}J}_{L^2(\mathbb{T}^3)}\bigg| \\
  &\leq \norm{\nabla h}_{L^\infty} \norm{\omega}_{X_{r,\tau,s}}\|\Lambda_m^r e^{\tau\Lambda_m^{1/s}}J\|_{L^2}.
\end{align*}
In order to estimate ${T}_{u,J,J}^{(2)}$, a little different from Lemma 2.5 of \cite{KV1}, we rewrite it into the sum of the following three terms,
\begin{align}\label{3.7}
  &{T}^{(2)}_{u,J,J}\notag\\
  &=i(2\pi)^3\sum_{j+k+\ell=0}\bigg[(\hat{u}_j\cdot k)\abs{\ell_m}^{r-\frac{1}{2s}}\bigg(e^{\tau(\abs{\ell_m}^{1/s}-\abs{k_m}^{1/s})}-1\notag\\
  &\quad\quad\quad\quad\quad\quad\quad-\tau(\abs{\ell_m}^{1/s}-\abs{k_m}^{1/s})\bigg) e^{\tau\abs{k_m}^{1/s}}(\hat{J}_k\cdot\hat{J}_\ell)\abs{\ell_m}^{r+\frac{1}{2s}} e^{\tau\abs{\ell_m}^{1/s}}\bigg] \notag\\
  &\quad+i(2\pi)^3\sum_{j+k+\ell=0}\bigg[\tau(\abs{\ell_m}^{r+\frac{1}{2s}}-\abs{k_m}^{r+\frac{1}{2s}})e^{\tau\abs{k_m}^{1/s}}(\hat{u}_j\cdot k)(\hat{J}_k\cdot\hat{J}_\ell) \notag\\
  &\quad\quad\quad\quad\quad\quad\times\abs{\ell_m}^{r+\frac{1}{2s}}e^{\tau\abs{\ell_m}^{1/s}}\bigg] \notag\\
  &\quad-i(2\pi)^3\sum_{j+k+\ell=0}\bigg[\tau\abs{k_m}^{1/s}(\abs{\ell_m}^{r-\frac{1}{2s}}-\abs{k_m}^{r-\frac{1}{2s}})e^{\tau\abs{k_m}^{1/s}}(\hat{u}_j\cdot k)\notag\\
  &\quad\quad\quad\quad\quad\quad\times(\hat{J}_k\cdot\hat{J}_\ell)\abs{\ell_m}^{r+\frac{1}{2s}}e^{\tau\abs{\ell_m}^{1/s}}\bigg]\notag\\
  &:={R}_{u,J,J}^{(1)}+{R}_{u,J,J}^{(2)}-{R}_{u,J,J}^{(3)}.
\end{align}
We remark that we may have a different form of the above expression if $s=1$, see \cite{KV1}, however the above identity is valid for all $s\geq1$.
For the first term ${R}_{u,J,J}^{(1)}$, we appeal to the inequality $\abs{e^\xi-1-\xi}\leq \xi^2 e^{\abs{\xi}}$, for $\xi=\tau(|\ell_m|^{1/s}-|k_m|^{1/s})\in\mathbb{R}$, the triangle inequality $\abs{\ell_m}^{r-\frac{1}{2s}}\leq C(\abs{j_m}^{r-\frac{1}{2s}}+\abs{k_m}^{r-\frac{1}{2s}})$ and
\begin{equation}\label{3.8}
 \abs{\abs{\ell_m}^{1/s}-\abs{k_m}^{1/s}}\leq \abs{j_m}^{1/s},\quad \abs{\abs{\ell_m}^{1/s}-\abs{k_m}^{1/s}}\leq C\frac{\abs{j_m}}{\abs{\ell_m}^{1-1/s}+\abs{k_m}^{1-1/s}},
\end{equation}
where we note that $\abs{\ell_m}^{1-1/s}+\abs{k_m}^{1-1/s}\neq 0$.
With use of the above inequalities, ${R}^{(1)}_{u,J,J}$ can be bounded by
\begin{align}\label{3.9}
  &\abs{{R}^{(1)}_{u,J,J}}\notag\\
  &\leq C\tau^2\sum_{j+k+\ell=0}\bigg[|\hat{u}_j|\abs{k}_1(\abs{j_m}^{r-\frac{1}{2s}}+\abs{k_m}^{r-\frac{1}{2s}})\abs{j_m}^{1/s}\frac{\abs{j_m}}{\abs{\ell_m}^{1-1/s}+\abs{k_m}^{1-1/s}} \notag\\
  &\quad\quad\quad\quad\quad\quad\quad\quad\times e^{\tau\abs{j_m}^{1/s}} e^{\tau\abs{k_m}^{1/s}}|\hat{J}_k|\abs{\ell_m}^{r+\frac{1}{2s}}e^{\tau\abs{\ell_m}^{1/s}}|\hat{J}_\ell|\bigg] \notag\\
  &\leq C\tau^2\sum_{j+k+\ell=0}\bigg[\big(\abs{j_m}^{r+\frac{1}{2s}+1}e^{\tau\abs{j_m}^{1/s}}|\hat{u}_j| \big)\big(\abs{k}_1 e^{\tau\abs{k_m}^{1/s}}|\hat{J}_k|\big)\notag\\
  &\quad\quad\quad\quad\quad\times(\abs{\ell_m}^{r+\frac{1}{2s}}e^{\tau\abs{\ell_m}^{1/s}}|\hat{J}_\ell|)\bigg] \notag\\
  &\quad+C\tau^2\sum_{j+k+\ell=0}\bigg[\big(\abs{j_m}^{1+\frac{1}{s}}e^{\tau\abs{j_m}^{1/s}}|\hat{u}_j| \big)\big(\abs{k}_1 \abs{k_m}^{r-\frac{1}{2s}}\frac{1}{\abs{\ell_m}^{1-1/s}+\abs{k_m}^{1-1/s}}\notag\\
  &\quad\quad\quad\quad\quad\quad\quad\quad\quad\times e^{\tau\abs{k_m}^{1/s}}|\hat{J}_k|\big)\big(\abs{\ell_m}^{r+\frac{1}{2s}}e^{\tau\abs{\ell_m}^{1/s}}|\hat{J}_\ell|\big)\bigg]\notag\\
  &\leq C\tau^2 \norm{\omega}_{Y_{r,\tau,s}}\norm{J}_{X_{r,\tau,s}} \norm{\Lambda_m^{r+\frac{1}{2s}}e^{\tau\Lambda_m^{1/s}} J}_{L^2} \notag\\
  &\quad +C\tau^2\norm{\omega}_{X_{r,\tau,s}}\norm{J}_{Y_{r,\tau,s}}\norm{\Lambda_m^{r+\frac{1}{2s}}e^{\tau\Lambda_m^{1/s}} J}_{L^2}.
\end{align}

In order to estimate the second term ${R}^{(2)}_{u,J,J}$, we use the mean value theorem again.
There exists a constant $\tilde{\theta}_{m,k,\ell}\in (0,1)$ such that
\begin{align}\label{3.10}
  &\abs{\ell_m}^{r+\frac{1}{2s}}-\abs{k_m}^{r+\frac{1}{2s}} \notag\\
  &=(r+\frac{1}{2s})(\abs{\ell_m}-\abs{k_m})\big[ (\tilde{\theta}_{m,k,\ell}\abs{\ell_m}+(1-\tilde{\theta}_{m,k,\ell})\abs{k_m})^{r+\frac{1}{2s}-1}-\abs{k_m}^{r+\frac{1}{2s}-1}\big] \notag\\
  &\quad+(r+\frac{1}{2s})(\abs{\ell_m}-\abs{k_m})\abs{k_m}^{r+\frac{1}{2s}-1}.
\end{align}
The first term on the right side of \eqref{3.10} is bounded by
$C\abs{j_m}^2(\abs{j_m}^{r-2+\frac{1}{2s}}+\abs{k_m}^{r-2+\frac{1}{2s}})$ for some constant $C$ depending on $r, s$.
For the latter term we use the decomposition \eqref{3.5} again,
and note in the region $\{\sgn(k_m+j_m)\sgn(k_m)=-1\}$ we have $\abs{k_m}\leq\abs{j_m}$ and $e^{\tau\abs{k_m}^{1/s}}\leq 1+\tau\abs{j_m}^{1/s}e^{\tau\abs{j_m}^{1/s}}$. Combining these facts, we have
\begin{align}\label{3.11}
  &\abs{{R}^{(2)}_{u,J,J}}\notag\\
  &\leq C\tau\sum_{j+k+\ell=0}\bigg[\abs{j_m}^2(\abs{j_m}^{r-2+\frac{1}{2s}}+\abs{k_m}^{r-2+\frac{1}{2s}})(1+\tau\abs{k_m}^{1/s}e^{\tau\abs{k_m}^{1/s}})\notag\\
  &\quad\quad\quad\quad\quad\quad\qquad\times |\hat{u}_j|\abs{k}_1 |\hat{J}_k|\abs{\ell_m}^{r+\frac{1}{2s}}e^{\tau\abs{\ell_m}^{1/s}}|\hat{J}_\ell|\bigg] \notag\\
  &\quad+ C\tau\abs{\inner{\partial_m u\cdot\nabla H_m\Lambda_m^{r-1+\frac{1}{2s}}e^{\tau\Lambda_m^{1/s}}J,
  \Lambda_m^{r+\frac{1}{2s}}e^{\tau\Lambda_m^{1/s}}J}_{L^2(\mathbb{T}^3)}}\notag\\
  &\quad+C\tau\sum_{j+k+\ell=0}\abs{j_m}^{r+\frac{1}{2s}}|\hat{u}_j|(1+\tau\abs{j_m}^{1/s}e^{\tau\abs{j_m}^{1/s}})\abs{k}_1 |\hat{J}_k|\abs{\ell_m}^{r+\frac{1}{2s}}e^{\tau\abs{\ell_m}^{1/s}}|\hat{J}_\ell|\notag\\
  &\leq C\tau\norm{\omega}_{H^r}\norm{J}_{H^r}\norm{\Lambda_m^r e^{\tau\Lambda_m^{1/s}} J}_{L^2}+C\tau\norm{\nabla u}_{L^\infty}\norm{J}_{Y_{r,\tau,s}}\norm{\Lambda_m^{r+\frac{1}{2s}}e^{\tau\Lambda_m^{1/s}}J}_{L^2}\notag\\
  &\quad +C\tau^2\norm{\omega}_{H^r}\norm{J}_{Y_{r,\tau,s}}^2 +C\tau^2 \norm{J}_{H^r}\norm{\omega}_{Y_{r,\tau,s}}\norm{J}_{Y_{r,\tau,s}},
\end{align}
where we have used $\abs{\ell_m}^{\frac{1}{2s}}\leq \abs{j_m}^{\frac{1}{2s}}+\abs{k_m}^{\frac{1}{2s}}$ for the estimate of the first term. In order to estimate the third term ${R}^{(3)}_{u,J,J}$, we once again expand the $\abs{\ell_m}^{r-\frac{1}{2s}}-\abs{k_m}^{r-\frac{1}{2s}}$ by mean value theorem,
\begin{align}\label{3.12}
  & \abs{\ell_m}^{r-\frac{1}{2s}}-\abs{k_m}^{r-\frac{1}{2s}} \notag\\
  &=(r-\frac{1}{2s})(\abs{\ell_m}-\abs{k_m})\big[(\theta^\ast_{m,k,\ell} \abs{\ell_m}+(1-\theta^\ast_{m,k,\ell})\abs{k_m})^{r-\frac{1}{2s}-1}-\abs{k_m}^{r-\frac{1}{2s}-1}\big] \notag\\
  &\quad+(r-\frac{1}{2s})(\abs{\ell_m}-\abs{k_m})\abs{k_m}^{r-1-\frac{1}{2s}}.
\end{align}
Using similar method as above, ${R}^{(3)}_{u,J,J}$ can also be bounded by
\begin{align}\label{3.13}
  &\abs{R_{u,J,J}^{(3)}} \notag\\
  &\leq C\sum_{j+k+\ell=0}\tau\bigg[\abs{k_m}^{1/s}\abs{j_m}^2(\abs{j_m}^{r-\frac{1}{2s}-2}+\abs{k_m}^{r-\frac{1}{2s}-2})(1+\tau\abs{k_m}^{1/s}e^{\tau\abs{k_m}^{1/s}})\notag\\
  &\qquad\qquad\qquad\times |\hat{u}_j|\abs{k}_1 |\hat{J}_k| |\hat{J}_\ell|\abs{\ell_m}^{r+\frac{1}{2s}}e^{\tau\abs{\ell_m}^{1/s}}\bigg]  \notag\\
  &\quad+C\sum_{j+k+\ell=0}\tau\abs{j_m}^{r+\frac{1}{2s}}(1+\tau\abs{j_m}^{1/s}e^{\tau\abs{j_m}^{1/s}}) |\hat{u}_j|\abs{k}_1 |\hat{J}_k| |\hat{J}_\ell|\abs{\ell_m}^{r+\frac{1}{2s}}e^{\tau\abs{\ell_m}^{1/s}}\notag\\
  &\quad+C\tau\abs{\inner{\partial_m u\cdot\nabla H_m\Lambda_m^{r+\frac{1}{2s}-1}e^{\tau\Lambda_m^{1/s}}J,
  \Lambda_m^{r+\frac{1}{2s}}e^{\tau\Lambda_m^{1/s}}J}_{L^2(\mathbb{T}^3)}}\notag\\
  &\leq C\tau \norm{\omega}_{H^r}\norm{J}_{H^r}\norm{\Lambda_m^r e^{\tau \Lambda_m^{1/s}} J}_{L^2}+C\tau \norm{\nabla u}_{L^\infty}\norm{J}_{Y_{r,\tau,s}}\norm{\Lambda_m^{r+\frac{1}{2s}} e^{\tau \Lambda_m^{1/s}} J}_{L^2}\notag\\
  &\quad +C\tau^2\norm{\omega}_{H^r}\norm{J}_{Y_{r,\tau,s}}\norm{\Lambda_m^{r+\frac{1}{2s}} e^{\tau \Lambda_m^{1/s}} J}_{L^2} \notag\\
  &\quad +C\tau^2 \norm{J}_{H^r}\norm{\omega}_{Y_{r,\tau,s}}\norm{\Lambda_m^{r+\frac{1}{2s}} e^{\tau \Lambda_m^{1/s}} J}_{L^2},
\end{align}
where we also used $|\ell_m|^{\frac{1}{2s}}\leq |j_m|^{\frac{1}{2s}}+|k_m|^{\frac{1}{2s}}$ for the estimate of the first term and here $C$ is a constant depending on $r, s$ for $r>\frac{5}{2}+\frac{3}{2s}$. Combining \eqref{3.9}, \eqref{3.11}, \eqref{3.13} and the estimate \eqref{3.6} on $T_{u,J,J}^{(1)}$ in \eqref{3.3}, we have proven that the term $|(u\cdot\nabla J,\Lambda_m^{2r}e^{2\tau\Lambda_m^{1/s}}J)_{L^2(\mathbb{T}^3)}|$ is bounded by the right of \eqref{3.2}. 

In order to estimate the coupled term $(J\cdot\nabla u,\Lambda_m^{2r}e^{2\tau\Lambda_m^{1/s}}J)_{L^2(\mathbb{T}^3)}$, we treat it as follows.
First of all, we note that $\norm{\omega}_{L^\infty}\leq \norm{\nabla u}_{L^\infty}$, $\norm{J}_{L^\infty}\leq \norm{\nabla h}_{L^\infty}$ and
\begin{align}\label{3.14}
  & \abs{\inner{\Lambda_m^r e^{\tau\Lambda_m^{1/s}}J\cdot\nabla u,\Lambda_m^r e^{\tau\Lambda_m^{1/s}}J}_{L^2(\mathbb{T}^3)}}+\abs{\inner{J\cdot\nabla\Lambda_m^r e^{\tau\Lambda_m^{1/s}}u,\Lambda_m^r e^{\tau \Lambda_m^{1/s}}J}_{L^2(\mathbb{T}^3)}} \notag\\
   &\leq C\norm{\nabla u}_{L^\infty}\norm{\Lambda_m^r e^{\tau\Lambda_m^{1/s}}J}_{L^2}^2+C\norm{\nabla h}_{L^\infty}\norm{\omega}_{X_{r,\tau,s}}\norm{\Lambda_m^r e^{\tau\Lambda_m^{1/s}}J}_{L^2}.
\end{align}
Then we substract $(J\cdot\nabla u,\Lambda_m^{2r}e^{2\tau\Lambda_m^{1/s}}J)_{L^2(\mathbb{T}^3)}$ by
$$
\inner{\Lambda_m^r e^{\tau\Lambda_m^{1/s}}J\cdot\nabla u,\Lambda_m^r e^{\tau\Lambda_m^{1/s}}J}_{L^2(\mathbb{T}^3)}+\inner{J\cdot\nabla \Lambda_m^r e^{\tau\Lambda_m^{1/s}} u,\Lambda_m^r e^{\tau\Lambda_m^{1/s}}J}_{L^2(\mathbb{T}^3)}
$$
and we consider their differences
\begin{align}\label{3.15}
  &\inner{J\cdot\nabla u,\Lambda_m^{2r}e^{2\tau\Lambda_m^{1/s}}J}_{L^2(\mathbb{T}^3)}-\inner{\Lambda_m^r e^{\tau\Lambda_m^{1/s}}J\cdot\nabla u,\Lambda_m^r e^{\tau\Lambda_m^{1/s}}J} _{L^2(\mathbb{T}^3)}\notag\\
  &\quad-\inner{J\cdot\nabla \Lambda_m^r e^{\tau\Lambda_m^{1/s}}u,\Lambda_m^r e^{\tau\Lambda_m^{1/s}}J}_{L^2(\mathbb{T}^3)} \notag\\
  &=i(2\pi)^3\sum_{j+k+\ell=0}(\abs{\ell_m}^r-\abs{j_m}^r)(e^{\tau\abs{\ell_m}^{1/s}}-e^{\tau\abs{k_m}^{1/s}})(\hat{J}_j\cdot k)(\hat{u}_k\cdot\hat{J}_\ell)\abs{\ell_m}^r e^{\tau\abs{\ell_m}^{1/s}} \notag\\
  &\quad+i(2\pi)^3\sum_{j+k+\ell=0}(\abs{\ell_m}^r-\abs{k_m}^r-\abs{j_m}^r)e^{\tau\abs{k_m}^{1/s}}(\hat{J}_j\cdot k)(\hat{u}_k\cdot\hat{J}_\ell)\abs{\ell_m}^r e^{\tau\abs{\ell_m}^{1/s}} \notag\\
  &\quad+i(2\pi)^3\sum_{j+k+\ell=0}\abs{j_m}^r(e^{\tau\abs{\ell_m}^{1/s}}-e^{\tau\abs{j_m}^{1/s}})(\hat{J}_j\cdot k)(\hat{u}_k\cdot\hat{J}_\ell)\abs{\ell_m}^r e^{\tau\abs{\ell_m}^{1/s}} \notag\\
  &:=\mathcal{T}_{J,u,J}^{(1)}+\mathcal{T}_{J,u,J}^{(2)}+\mathcal{T}_{J,u,J}^{(3)}.
\end{align}
It rested to estimate the right hand side of \eqref{3.15}. For the first term $\mathcal{T}_{J,u,J}^{(1)}$, we appeal to the mean value theorem for $\abs{\ell_m}^r-\abs{j_m}^r$, and $\abs{e^\xi- 1}\leq\abs{\xi} e^{\abs{\xi}}$, for $\xi=\tau(|\ell_m|^{1/s}-|k_m|^{1/s})\in\mathbb{R}$, and the inequality \eqref{3.8},
\begin{align}\label{3.16}
  \big|(\abs{\ell_m}^r-\abs{j_m}^r) &(e^{\tau\abs{\ell_m}^{1/s}}-e^{\tau\abs{k_m}^{1/s}}) \big| \notag\\
  &\leq C\tau \frac{\abs{j_m}\abs{k_m}^r+\abs{k_m}\abs{j_m}^r}{\abs{\ell_m}^{1-1/s}+\abs{k_m}^{1-1/s}} e^{\tau\abs{j_m}^{1/s}}e^{\tau\abs{k_m}^{1/s}} \notag\\
  &\leq C\tau |j_m|^{1/s}|k_m|^r e^{\tau |j_m|^{1/s}}(1+\tau|k_m|^{1/s}e^{\tau|k_m|^{1/s}}) \notag\\
  &\quad+C\tau |j_m|^r |k_m|^{1/s}(1+\tau |j_m|^{1/s}e^{\tau |j_m|^{1/s}})e^{\tau |k_m|^{1/s}}.
\end{align}
Substituting the right of \eqref{3.16} into $\mathcal{T}_{J,u,J}^{(1)}$ and using again the inequality $e^{\tau|j_m|^{1/s}}\leq 1+\tau|j_m|^{1/s}e^{\tau|j_m|^{1/s}}$ and $e^{\tau|k_m|^{1/s}}\leq 1+\tau|k_m|^{1/s}e^{\tau|k_m|^{1/s}}$ for the order-$\tau$ term, we have
\begin{align}\label{3.17}
  \abs{\mathcal{T}_{J,u,J}^{(1)}} & \leq C\tau\norm{\omega}_{H^r}\norm{J}_{H^r}\norm{\Lambda_m^r e^{\tau\Lambda_m^{1/s}}J}_{L^2}+C\tau^2 \norm{\omega}_{H^r}\norm{J}_{Y_{r,\tau,s}}^2 \notag\\
  &\quad+C\tau^2 \norm{J}_{H^r}\norm{\omega}_{Y_{r,\tau,s}}\norm{J}_{Y_{r,\tau,s}} +C\tau^2 \norm{J}_{X_{r,\tau,s}}\norm{\omega}_{Y_{r,\tau,s}}\norm{\Lambda_m^{r+\frac{1}{2s}}e^{\tau\Lambda_m^{1/s}}J}_{L^2} \notag\\
  &\quad+C\tau^2 \norm{\omega}_{X_{r,\tau,s}}\norm{J}_{Y_{r,\tau,s}}^2,
\end{align}
where we used the inequalities $|k_m|^{\frac{1}{2s}}\leq |j_m|^{\frac{1}{2s}}+|\ell_m|^{\frac{1}{2s}}$ and $|j_m|^{\frac{1}{2s}}\leq |k_m|^{\frac{1}{2s}}+|\ell_m|^{\frac{1}{2s}}$.
For the second term $\mathcal{T}_{J,u,J}^{(2)}$, by the mean value theorem we have
\begin{equation*}%\label{3.28}
  \abs{(\abs{\ell_m}^r-\abs{k_m}^r)-\abs{j_m}^r}\leq C\abs{j_m}(\abs{j_m}^{r-1}+\abs{k_m}^{r-1})+\abs{j_m}^r.
\end{equation*}
Using the inequality $e^x\leq e+x^2 e^x$, for all $x=\tau|k_m|^{1/s}$, then we obtain
\begin{align}\label{3.18}
  \abs{\mathcal{T}_{J,u,J}^{(2)}} &\leq C\norm{J}_{H^r}\norm{\omega}_{H^r} \norm{\Lambda_m^r e^{\tau\Lambda_m^{1/s}}J}_{L^2} \notag\\
  &\quad+C\tau^2 \norm{J}_{H^r}\norm{\omega}_{Y_{r,\tau,s}}\norm{\Lambda_m^{r+\frac{1}{2s}} e^{\tau\Lambda_m^{1/s}}J}_{L^2},
\end{align}
where we used $\abs{k_m}^{\frac{1}{2s}}\leq \abs{j_m}^{\frac{1}{2s}}+\abs{\ell_m}^{\frac{1}{2s}}$ in the estimate of the second term on right of \eqref{3.18}. For the third term $\mathcal{T}_{J,u,J}^{(3)}$, we use the inequality $\abs{e^\xi-1}\leq \abs{\xi}e^{\abs{\xi}}$, for $\xi=\tau(|\ell_m|^{1/s}-|j_m|^{1/s})\in\mathbb{R}$, and the inequality $e^\xi\leq 1+\xi e^\xi$, for $\xi=\tau|j_m|^{1/s}$ and $\xi=\tau|k_m|^{1/s}$, and the triangle inequality $\abs{j_m}^{\frac{1}{2s}}\leq \abs{k_m}^{\frac{1}{2s}}+\abs{\ell_m}^{\frac{1}{2s}}$.
Thus we finally have
\begin{align}\label{3.19}
  \abs{\mathcal{T}_{J,u,J}^{(3)}} &\leq C\tau\norm{J}_{H^r}\norm{\omega}_{H^r} \norm{\Lambda_m^r e^{\tau\Lambda_m^{1/s}}J}_{L^2}+C\tau^2\norm{\omega}_{H^r}\norm{J}_{Y_{r,\tau,s}}^2 \notag\\
  &\quad+ C\tau^2\norm{\omega}_{X_{r,\tau,s}}
  \norm{\Lambda_m^{r+\frac{1}{2s}}e^{\tau\Lambda_m^{1/s}}J}_{L^2}^2.
\end{align}
Collecting \eqref{3.17}, \eqref{3.18}, \eqref{3.19} and \eqref{3.14}, we have the estimate
\begin{align}\label{3.20}
  &\abs{\inner{J\cdot\nabla u,\Lambda_m^{2r}e^{2\tau\Lambda_m^{1/s}}J}_{L^2(\mathbb{T}^3)}}\notag\\
  & \leq C(\norm{\nabla u}_{L^\infty}\norm{J}_{X_{r,\tau,s}}+\norm{\nabla h}_{L^\infty}\norm{\omega}_{X_{r,\tau,s}}) \norm{J}_{X_{r,\tau,s}}\notag \\
  &\quad+ C\tau\norm{\Psi}_{H^r}^2\norm{\Lambda_m^r e^{\tau\Lambda_m^{1/s}}J}_{L^2}
  + C\tau^2(\norm{\Psi}_{H^r}+\norm{\Psi}_{X_{r,\tau,s}})\norm{\Psi}_{Y_{r,\tau,s}}\norm{J}_{Y_{r,\tau,s}}.
\end{align}
Obviously the right of \eqref{3.20} is also bounded by the right of \eqref{3.2}, thus the proof is complete.
\end{proof}

In the following, we give the main Lemma concerning the estimates of the coupled nonlinear terms.
\begin{lemma}\label{Lemma 3.6}
Let $m=1,2,3$. Let $\tau>0$, $r>\frac{5}{2}+\frac{3}{2s}$, and $u=\mathcal{K}* \omega, h=\mathcal{K}* J$ with $\omega, J\in Y_{r,\tau,s}$. Then we have the following upper bounded estimates :
   \begin{align}\label{3.21}
          &\abs{\inner{h\cdot\nabla J,\Lambda_m^{2r} e^{2\tau \Lambda_m^{1/s}}\omega}_{L^2(\mathbb{T}^3)} +\inner{h\cdot\nabla \omega,\Lambda_m^{2r} e^{2\tau \Lambda_m^{1/s}}J}_{L^2(\mathbb{T}^3)}} \notag \\
          &\quad \leq C\big(\tau\norm{\nabla h}_{L^\infty}+\tau^2\norm{\Psi}_{{H}^r}+\tau^2\norm{\Psi}_{X_{r,\tau,s}} \big)\norm{\Psi}^2_{Y_{r,\tau,s}}\notag\\
          & \quad\quad + C\big(\norm{\nabla h}_{L^\infty}\norm{\Psi}_{X_{r,\tau,s}}+(1+\tau)\norm{\Psi}_{{H}^r}^2\big)\norm{\Psi}_{X_{r,\tau,s}},
    \end{align}
  \begin{align}\label{3.22}
           & \abs{\inner{J\cdot\nabla h,\Lambda_m^{2r} e^{2\tau \Lambda_m^{1/s}}\omega}_{L^2(\mathbb{T}^3)}}+\abs{\inner{\omega\cdot\nabla h,\Lambda_m^{2r} e^{2\tau \Lambda_m^{1/s}}J}_{L^2(\mathbb{T}^3)}} \notag\\
           & \leq C(\norm{\nabla u}_{L^\infty}+\norm{\nabla h}_{L^\infty})\norm{\Psi}_{X_{r,\tau,s}}^2+C\tau\norm{\Psi}_{{H}^r}^2\norm{\Psi}_{X_{r,\tau,s}}\notag\\
           &\qquad+C\tau^2(\norm{\Psi}_{{H}^r}+\norm{\Psi}_{X_{r,\tau,s}})\norm{\Psi}_{Y_{r,\tau,s}}^2,
  \end{align}
where $C$ is a constant depending only on $r,s$.
\end{lemma}
We note that the key point in the proof of Lemma \ref{Lemma 3.6} is that the coefficients of $\tau$ and $\tau^2$ are carefully arranged such that on one hand we can obtain an upper bound of $\norm{\omega}_{X_{t,\tau,s}}$, on the other hand we can obtain a lower bound of $\tau$ in terms of $\norm{\nabla u}_{L^\infty}$ and $\norm{\nabla h}_{L^\infty}$.

\begin{proof}[Proof of \eqref{3.21}]
Since $h=\mathcal{K} * J$ is divergence-free, we have the following cancellation property, by integration by parts and the symmetry structure, 
  $$
   \inner{h\cdot\nabla\Lambda_m^r e^{\tau\Lambda_m^{1/s}}J,\Lambda_m^r e^{\tau\Lambda_m^{1/s}}\omega}_{L^2(\mathbb{T}^3)}+\inner{h\cdot\nabla \Lambda_m^r e^{\tau\Lambda_m^{1/s}}\omega,\Lambda_m^r e^{\tau\Lambda_m^{1/s}}J}_{L^2(\mathbb{T}^3)}=0.
  $$
Thus we have 
\begin{align}\label{3.23}
  & \inner{h\cdot\nabla J,\Lambda_m^{2r} e^{2\tau\Lambda_m^{1/s}}\omega}_{L^2(\mathbb{T}^3)}+\inner{h\cdot\nabla \omega,\Lambda_m^{2r} e^{2\tau\Lambda_m^{1/s}}J} _{L^2(\mathbb{T}^3)}\notag\\
  &=\inner{h\cdot\nabla J,\Lambda_m^{2r} e^{2\tau\Lambda_m^{1/s}}\omega}_{L^2(\mathbb{T}^3)}+\inner{h\cdot\nabla \omega,\Lambda_m^{2r} e^{2\tau\Lambda_m^{1/s}}J}_{L^2(\mathbb{T}^3)}  \notag\\
  &\quad -\inner{h\cdot\nabla\Lambda_m^r e^{\tau\Lambda_m^{1/s}}J,\Lambda_m^r e^{\tau\Lambda_m^{1/s}}\omega}_{L^2(\mathbb{T}^3)}-\inner{h\cdot\nabla \Lambda_m^r e^{\tau\Lambda_m^{1/s}}\omega,\Lambda_m^r e^{\tau\Lambda_m^{1/s}}J}_{L^2(\mathbb{T}^3)} \notag\\
  &=i(2\pi)^3\sum_{ j+k+\ell=0}(\hat{h}_j\cdot k)(\abs{\ell_m}^r e^{\tau\abs{\ell_m}^{1/s}}-\abs{k_m}^r e^{\tau \abs{k_m}^{1/s}})(\hat{J}_k\cdot\hat{\omega}_\ell)\abs{\ell_m}^r e^{\tau\abs{\ell_m}^{1/s}}\notag \\
  &\quad+i(2\pi)^3\sum_{j+k+\ell=0}(\hat{h}_j\cdot k)(\abs{\ell_m}^r e^{\tau\abs{\ell_m}^{1/s}}-\abs{k_m}^r e^{\tau \abs{k_m}^{1/s}})(\hat{\omega}_k\cdot\hat{J}_\ell)\abs{\ell_m}^r e^{\tau\abs{\ell_m}^{1/s}}\notag\\
  &:=T_{h,J,\omega}+T_{h,\omega,J},
\end{align}
where the summation are taken over $\{j,k,\ell\in\mathbb{Z}^3;j+k+\ell=0, \ell_m\neq0, j,k,\ell\neq 0\}$ and we will sometimes use this property without mentioning it in the following.
Due to the symmetry of $T_{h,J,\omega}$ and $T_{h,\omega,J}$ on the right hand side of \eqref{3.23}, it suffices to estimate one of them.
Let us consider for example ${T}_{h,\omega,J}$. It also can be split into the summation of two terms ${T}_{h,\omega,J}={T}_{h,\omega,J}^{(1)}+{T}_{h,\omega,J}^{(2)}$, where
\begin{align*}
  {T}^{(1)}_{h,\omega,J}= &i(2\pi)^3 \sum_{j+k+\ell=0}(\hat{h}_j\cdot k)(\abs{\ell_m}^r-\abs{k_m}^r)e^{\tau\abs{k_m}^{1/s}}(\hat{\omega}_k\cdot\hat{J}_\ell)\abs{\ell_m}^r e^{\tau\abs{\ell_m}^{1/s}},\\
  {T}^{(2)}_{h,\omega,J}= &i(2\pi)^3 \sum_{j+k+\ell=0}(\hat{h}_j\cdot k)\abs{\ell_m}^r (e^{\tau\abs{\ell_m}^{1/s}}-e^{\tau\abs{k_m}^{1/s}})(\hat{\omega}_k\cdot\hat{J}_\ell)\abs{\ell_m}^r e^{\tau\abs{\ell_m}^{1/s}}\,.
\end{align*}
In order to estimate ${T}^{(1)}_{h,\omega,J}$, we appeal to the expansion of \eqref{3.4}, \eqref{3.5} and the arguments of \eqref{3.6}. Then we immediately have %$\abs{\ell_m}^r-\abs{k_m}^r$ by means of mean value theorem,
%\begin{align*}%\label{3.10}
%  \abs{\ell_m}^r-\abs{k_m}^r &=r(\abs{\ell_m}-\abs{k_m})(\theta_{m,k,\ell}\abs{\ell_m}+(1-\theta_{m,k,\ell})\abs{k_m})^{r-1} \\
%  &=r(\abs{\ell_m}-\abs{k_m})\big[\big(\theta_{m,k,\ell}\abs{\ell_m}+
%  (1-\theta_{m,k,\ell})\abs{k_m}\big)^{r-1}-\abs{k_m}^{r-1}\big]\\
%  &\quad+r(\abs{\ell_m}-\abs{k_m})\abs{k_m}^{r-1},
%\end{align*}
%where $\theta_{m,k,\ell}\in(0,1)$ is a constant. Since $j+k+\ell=0$, we have, by the triangle inequality,
%\begin{equation*}%\label{3.11}
%\begin{aligned}
%  &\quad\abs{r(\abs{\ell_m}-\abs{k_m})\big[(\theta_{m,k,\ell}\abs{\ell_m}+(1-\theta_{m,k,\ell})\abs{k_m})^{r-1}-\abs{k_m}^{r-1}\big]}\\
%  &\leq C\abs{j_m}^2(\abs{j_m}^{r-2}+\abs{k_m}^{r-2}).
%\end{aligned}
%\end{equation*}
%Since $j_m+k_m+\ell_m=0$, we have the following decomposition, introduced by \cite{KV1},
%\begin{equation}\label{3.12}
%\begin{split}
%  \abs{\ell_m}-\abs{k_m} &=\abs{j_m+k_m}-\abs{k_m}\\
%   &=j_m\sgn(k_m)+2(j_m+k_m)\sgn(j_m)\chi_{\{\sgn(k_m+j_m)\sgn(k_m)=-1\}}.
%\end{split}
%\end{equation}
%In the region $\{\sgn(k_m+j_m)\sgn(k_m)=-1\}$, we have $\abs{k_m}\leq\abs{j_m}$. Then with use of $e^x\leq e+x^2 e^{x}$ for $x\geq0$ and Plancherel's theorem we have, by discrete Cauchy-Schwartz inequality,
\begin{align}\label{3.24}
  &\abs{{T}^{(1)}_{h,\omega,J}} \notag\\
  &\leq  C\sum_{j+k+\ell=0}\bigg\{(\abs{j_m}^{r}+\abs{j_m}^2\abs{k_m}^{r-2})(e+\tau^2\abs{k_m}^{2/s}e^{\tau\abs{k_m}^{1/s}})
  |\hat{h}_j|\abs{k}_1\abs{\hat{\omega}_k}|\hat{J}_\ell| \notag\\
  &\quad\quad\quad\quad\quad\quad\quad\quad\times\abs{\ell_m}^r e^{\tau\abs{\ell_m}^{1/s}}\bigg\} \notag \\
  &\quad+C\bigg|\sum_{j+k+\ell=0}j_m\sgn(k_m)\abs{k_m}^{r-1}e^{\tau\abs{k_m}^{1/s}}(\hat{h}_j\cdot k)(\hat{\omega}_k\cdot\hat{J}_\ell)\abs{\ell_m}^r e^{\tau\abs{\ell_m}^{1/s}}\bigg| \notag\\
  &\leq C\norm{\nabla h}_{L^\infty}\norm{\omega}_{X_{r,\tau,s}}\norm{J}_{X_{r,\tau,s}}+C\norm{\omega}_{{H}^r}\norm{J}_{{H}^r}\norm{J}_{X_{r,\tau,s}} \notag \\
  &\quad+C\tau^2\norm{J}_{{H}^r}\norm{\omega}_{Y_{r,\tau,s}}\norm{J}_{Y_{r,\tau,s}},
\end{align}
where $C$ is some constant depending on $r$. Still the supremum of gradient of $h$ on the right hand side of \eqref{3.24} come from the use of Plancherel's theorem as follows,
\begin{align*}
  &\quad\bigg|\sum_{j+k+\ell=0}j_m\sgn(k_m)\abs{k_m}^{r-1}e^{\tau\abs{k_m}^{1/s}}(\hat{h}_j\cdot k)(\hat{\omega}_k\cdot\hat{J}_\ell)\abs{\ell_m}^r e^{\tau\abs{\ell_m}^{1/s}}\bigg| \\
  &=\bigg|\inner{\partial_m h\cdot\nabla H_m\Lambda_m^{r-1}e^{\tau\Lambda_m^{1/s}}\omega,\Lambda_m^r e^{\tau \Lambda_m^{1/s}}J}_{L^2(\mathbb{T}^3)}\bigg| \\
  &\leq \norm{\nabla h}_{L^\infty} \norm{\omega}_{X_{r,\tau,s}}\|\Lambda_m^r e^{\tau\Lambda_m^{1/s}}J\|_{L^2}.
\end{align*}
To estimate ${T}_{h,\omega,J}^{(2)}$, like \eqref{3.7}, we rewrite it into the sum of the following three terms, 
\begin{align}\label{3.25}
  &{T}^{(2)}_{h,\omega,J}\notag\\
  &=i(2\pi)^3\sum_{j+k+\ell=0}\bigg[(\hat{h}_j\cdot k)\abs{\ell_m}^{r-\frac{1}{2s}}\bigg(e^{\tau(\abs{\ell_m}^{1/s}-\abs{k_m}^{1/s})}-1\notag\\
  &\quad\quad\quad\quad\quad\quad\quad-\tau(\abs{\ell_m}^{1/s}-\abs{k_m}^{1/s})\bigg) e^{\tau\abs{k_m}^{1/s}}(\hat{\omega}_k\cdot\hat{J}_\ell)\abs{\ell_m}^{r+\frac{1}{2s}} e^{\tau\abs{\ell_m}^{1/s}}\bigg] \notag\\
  &\quad+i(2\pi)^3\sum_{j+k+\ell=0}\bigg[\tau(\abs{\ell_m}^{r+\frac{1}{2s}}-\abs{k_m}^{r+\frac{1}{2s}})e^{\tau\abs{k_m}^{1/s}}(\hat{h}_j\cdot k)(\hat{\omega}_k\cdot\hat{J}_\ell) \notag\\
  &\quad\quad\quad\quad\quad\quad\times\abs{\ell_m}^{r+\frac{1}{2s}}e^{\tau\abs{\ell_m}^{1/s}}\bigg] \notag\\
  &\quad-i(2\pi)^3\sum_{j+k+\ell=0}\bigg[\tau\abs{k_m}^{1/s}(\abs{\ell_m}^{r-\frac{1}{2s}}-\abs{k_m}^{r-\frac{1}{2s}})e^{\tau\abs{k_m}^{1/s}}(\hat{h}_j\cdot k)\notag\\
  &\quad\quad\quad\quad\quad\quad\times(\hat{\omega}_k\cdot\hat{J}_\ell)\abs{\ell_m}^{r+\frac{1}{2s}}e^{\tau\abs{\ell_m}^{1/s}}\bigg]\notag\\
  &:={R}_{h,\omega,J}^{(1)}+{R}_{h,\omega,J}^{(2)}+{R}_{h,\omega,J}^{(3)}.
\end{align}
The three terms ${R}_{h,\omega,J}^{(1)}, R_{h,\omega,J}^{(2)}$ and $R_{h,\omega,J}^{(3)}$ on right of \eqref{3.25} are estimated with the same arguments with \eqref{3.9}, \eqref{3.11} and \eqref{3.13}, thus we immediately have from the arguments of \eqref{3.8} and \eqref{3.9}, %we appeal to the inequality $\abs{e^x-1-x}\leq x^2 e^{\abs{x}}$, for all $x\in\mathbb{R}$, the triangle inequality $\abs{\ell_m}^{r-\frac{1}{2s}}\leq C(\abs{j_m}^{r-\frac{1}{2s}}+\abs{k_m}^{r-\frac{1}{2s}})$ and
%\begin{equation}\label{3.16}
% \abs{\abs{\ell_m}^{1/s}-\abs{k_m}^{1/s}}\leq \abs{j_m}^{1/s},\quad \abs{\abs{\ell_m}^{1/s}-\abs{k_m}^{1/s}}\leq C\frac{\abs{j_m}}{\abs{\ell_m}^{1-1/s}+\abs{k_m}^{1-1/s}},
%\end{equation}
%where we note that $\abs{\ell_m}^{1-1/s}+\abs{k_m}^{1-1/s}\neq 0$.
%With use of the above inequalities, ${R}^{(1)}_{h,\omega,J}$ can be bounded by
\begin{align}\label{3.26}
  &\abs{{R}^{(1)}_{h,\omega,J}}\notag\\
  &\leq C\tau^2\sum_{j+k+\ell=0}\bigg[|\hat{h}_j|\abs{k}_1(\abs{j_m}^{r-\frac{1}{2s}}+\abs{k_m}^{r-\frac{1}{2s}})\abs{j_m}^{1/s}\frac{\abs{j_m}}{\abs{\ell_m}^{1-1/s}+\abs{k_m}^{1-1/s}} \notag\\
  &\quad\quad\quad\quad\quad\quad\quad\quad\times e^{\tau\abs{j_m}^{1/s}} e^{\tau\abs{k_m}^{1/s}}\abs{\hat{\omega}_k}\abs{\ell_m}^{r+\frac{1}{2s}}e^{\tau\abs{\ell_m}^{1/s}}|\hat{J}_\ell|\bigg] \notag\\
  &\leq C\tau^2\sum_{j+k+\ell=0}\bigg[\big(\abs{j_m}^{r+\frac{1}{2s}+1}e^{\tau\abs{j_m}^{1/s}}|\hat{h}_j| \big)\big(\abs{k}_1 e^{\tau\abs{k_m}^{1/s}}\abs{\hat{\omega}_k}\big)\notag\\
  &\quad\quad\quad\quad\quad\times(\abs{\ell_m}^{r+\frac{1}{2s}}e^{\tau\abs{\ell_m}^{1/s}}|\hat{J}_\ell|)\bigg] \notag\\
  &\quad+C\tau^2\sum_{j+k+\ell=0}\bigg[\big(\abs{j_m}^{1+\frac{1}{s}}e^{\tau\abs{j_m}^{1/s}}|\hat{h}_j| \big)\big(\abs{k}_1 \abs{k_m}^{r-\frac{1}{2s}}\frac{1}{\abs{\ell_m}^{1-1/s}+\abs{k_m}^{1-1/s}}\notag\\
  &\quad\quad\quad\quad\quad\quad\quad\quad\quad\times e^{\tau\abs{k_m}^{1/s}}\abs{\hat{\omega}_k}\big)\big(\abs{\ell_m}^{r+\frac{1}{2s}}e^{\tau\abs{\ell_m}^{1/s}}|\hat{J}_\ell|\big)\bigg]\notag\\
  &\leq C\tau^2 \norm{J}_{Y_{r,\tau,s}}^2\norm{\omega}_{X_{r,\tau,s}}+C\tau^2\norm{J}_{X_{r,\tau,s}}\norm{\omega}_{Y_{r,\tau,s}}\norm{J}_{Y_{r,\tau,s}},
\end{align}
where $C$ is a appropriate constant.
%In order to estimate the second term ${R}^{(2)}_{h,\omega,J}$, we use the mean value theorem again.
%There exists a constant $\tilde{\theta}_{m,k,\ell}\in (0,1)$ such that
%\begin{align}\label{3.18}
%  &\abs{\ell_m}^{r+\frac{1}{2s}}-\abs{k_m}^{r+\frac{1}{2s}} \notag\\
%  &=(r+\frac{1}{2s})(\abs{\ell_m}-\abs{k_m})\big[ (\tilde{\theta}_{m,k,\ell}\abs{\ell_m}+(1-\tilde{\theta}_{m,k,\ell})\abs{k_m})^{r+\frac{1}{2s}-1}-\abs{k_m}^{r+\frac{1}{2s}-1}\big] \notag\\
%  &\quad+(r+\frac{1}{2s})(\abs{\ell_m}-\abs{k_m})\abs{k_m}^{r+\frac{1}{2s}-1}.
%\end{align}
%The first term on the right side of \eqref{3.18} is bounded by
%$C\abs{j_m}^2(\abs{j_m}^{r-2+\frac{1}{2s}}+\abs{k_m}^{r-2+\frac{1}{2s}})$ for some constant $C$ depending on $r, s$.
%For the latter term we use the decomposition \eqref{3.12} again,
%and note in the region $\{\sgn(k_m+j_m)\sgn(k_m)=-1\}$ we have $\abs{k_m}\leq\abs{j_m}$ and $e^{\tau\abs{k_m}^{1/s}}\leq 1+\tau\abs{j_m}^{1/s}e^{\tau\abs{j_m}^{1/s}}$. Combining these facts, we have
By use of the expansion \eqref{3.10} and similar arguments as \eqref{3.11}, we have
\begin{align}\label{3.27}
  &\abs{{R}^{(2)}_{h,\omega,J}}\notag\\
  &\leq C\tau\abs{\inner{\partial_m h\cdot\nabla H_m\Lambda_m^{r-1+\frac{1}{2s}}e^{\tau\Lambda_m^{1/s}}\omega,
  \Lambda_m^{r+\frac{1}{2s}}e^{\tau\Lambda_m^{1/s}}J}_{L^2(\mathbb{T}^3)}}\notag\\
  &\quad+C\tau\sum_{j+k+\ell=0}\bigg[\abs{j_m}^2(\abs{j_m}^{r-2+\frac{1}{2s}}+\abs{k_m}^{r-2+\frac{1}{2s}})(1+\tau\abs{k_m}^{1/s}e^{\tau\abs{k_m}^{1/s}})\notag\\
  &\quad\quad\quad\quad\quad\quad\qquad\times |\hat{h}_j|\abs{k}_1\abs{\hat{\omega}_k}\abs{\ell_m}^{r+\frac{1}{2s}}e^{\tau\abs{\ell_m}^{1/s}}|\hat{J}_\ell|\bigg] \notag\\
  &\quad+C\tau\sum_{j+k+\ell=0}\abs{j_m}^{r+\frac{1}{2s}}|\hat{h}_j|(1+\tau\abs{j_m}^{1/s}e^{\tau\abs{j_m}^{1/s}})\abs{k}_1\abs{\hat{\omega}_k}\abs{\ell_m}^{r+\frac{1}{2s}}e^{\tau\abs{\ell_m}^{1/s}}|\hat{J}_\ell|\notag\\
  &\leq C\tau\norm{J}_{H^r}\norm{\omega}_{H^r}\norm{J}_{X_{r,\tau,s}}+C\tau\norm{\nabla h}_{L^\infty}\norm{\omega}_{Y_{r,\tau,s}}\norm{J}_{Y_{r,\tau,s}}\notag\\
  &\quad +C\tau^2 \norm{J}_{H^r}\norm{\omega}_{Y_{r,\tau,s}}\norm{J}_{Y_{r,\tau,s}}+C\tau^2\norm{\omega}_{H^r}\norm{J}_{Y_{r,\tau,s}}^2,
\end{align}
where $C$ is a positive constant.
%where we used $\abs{\ell_m}^{\frac{1}{2s}}\leq \abs{j_m}^{\frac{1}{2s}}+\abs{k_m}^{\frac{1}{2s}}$ and Lemma \ref{Lemma 3.2}. In order to estimate the third term ${R}^{(3)}_{h,\omega,J}$, we once again expand the $\abs{\ell_m}^{r-\frac{1}{2s}}-\abs{k_m}^{r-\frac{1}{2s}}$ by mean value theorem,
%\begin{align*}%\label{3.20}
%  & \abs{\ell_m}^{r-\frac{1}{2s}}-\abs{k_m}^{r-\frac{1}{2s}} \\
%  &=(r-\frac{1}{2s})(\abs{\ell_m}-\abs{k_m})\big[(\theta^\ast_{m,k,\ell} \abs{\ell_m}+(1-\theta^\ast_{m,k,\ell})\abs{k_m})^{r-\frac{1}{2s}-1}-\abs{k_m}^{r-\frac{1}{2s}-1}\big] \notag\\
%  &\quad+(r-\frac{1}{2s})(\abs{\ell_m}-\abs{k_m})\abs{k_m}^{r-1-\frac{1}{2s}}
%\end{align*}
%Using similar method as above, ${R}^{(3)}_{h,\omega,J}$ can also be bounded by
By use of the expansion \eqref{3.12} and similar arguments as \eqref{3.13}, we have
\begin{align}\label{3.28}
  &\abs{{R}^{(3)}}_{h,\omega,J} \notag\\
  &\leq C\sum_{j+k+\ell=0}\tau\bigg[\abs{k_m}^{1/s}\abs{j_m}^2(\abs{j_m}^{r-\frac{1}{2s}-2}+\abs{k_m}^{r-\frac{1}{2s}-2})(1+\tau\abs{k_m}^{1/s}e^{\tau\abs{k_m}^{1/s}})\notag\\
  &\qquad\qquad\qquad\times |\hat{h}_j|\abs{k}_1 \abs{\hat{\omega}_k}|\hat{J}_\ell|\abs{\ell_m}^{r+\frac{1}{2s}}e^{\tau\abs{\ell_m}^{1/s}}\bigg]  \notag\\
  &\quad+C\sum_{j+k+\ell=0}\tau\abs{j_m}^{r+\frac{1}{2s}}(1+\tau\abs{j_m}^{1/s}e^{\tau\abs{j_m}^{1/s}}) |\hat{h}_j|\abs{k}_1 \abs{\hat{\omega}_k}|\hat{J}_\ell|\abs{\ell_m}^{r+\frac{1}{2s}}e^{\tau\abs{\ell_m}^{1/s}}\notag\\
  &\quad+C\tau\abs{\inner{\partial_m h\cdot\nabla H_m\Lambda_m^{r+\frac{1}{2s}-1}e^{\tau\Lambda_m^{1/s}}\omega,
  \Lambda_m^{r+\frac{1}{2s}}e^{\tau\Lambda_m^{1/s}}J}_{L^2(\mathbb{T}^3)}}\notag\\
  &\leq C\tau\norm{J}_{H^r}\norm{\omega}_{H^r}\norm{J}_{X_{r,\tau,s}}+C\tau \norm{\nabla h}_{L^\infty}\norm{\omega}_{Y_{r,\tau,s}}\norm{ J}_{Y_{r,\tau,s}}\notag\\
  &\quad+C\tau^2 \norm{J}_{H^r}\norm{\omega}_{Y_{r,\tau,s}}\norm{J}_{Y_{r,\tau,s}} +C\tau^2\norm{\omega}_{H^r}\norm{J}_{Y_{r,\tau,s}}^2,
\end{align}
where $C$ is a constant depending only on $r, s$ for $r>\frac{5}{2}+\frac{3}{2s}$.
Combining \eqref{3.26}, \eqref{3.27} and \eqref{3.28}, we have the estimate of ${T}_{h,\omega,J}^{(2)}$. Then with the estimate \eqref{3.24} of $T_{h,\omega,J}^{(1)}$, we have
\begin{align}\label{3.29}
  \abs{{T}_{h,\omega,J}} &\leq \abs{{T}^{(1)}_{h,\omega,J}}+\abs{{T}^{(2)}_{h,\omega,J}}  \notag\\
  & \leq  \abs{{T}^{(1)}_{h,\omega,J}}+\abs{{R}_{h,\omega,J}^{(1)}}+\abs{{R}_{h,\omega,J}^{(2)}}+\abs{{R}_{h,\omega,J}^{(3)}}\notag\\
  & \leq \big[C\norm{\nabla h}_{L^\infty}\norm{\Psi}_{X_{r,\tau,s}}+C(1+\tau)\norm{\Psi}_{H^r}^2\big]\norm{\Psi}_{X_{r,\tau,s}}\notag\\
  & \quad+\big[C\tau\norm{\nabla h}_{L^\infty}+C\tau^2(\norm{\Psi}_{H^r}+\norm{\Psi}_{X_{r,\tau,s}})\big]\norm{\Psi}_{Y_{r,\tau,s}}^2.
\end{align}
Symmetrically we have
\begin{align}\label{3.30}
  \abs{{T}_{h,J,\omega}} &\leq \abs{{T}^{(1)}_{h,J,\omega}}+\abs{{T}^{(2)}_{h,J,\omega}}  \notag\\
  & \leq  \abs{{T}^{(1)}_{h,J,\omega}}+\abs{{R}_{h,J,\omega}^{(1)}}+\abs{{R}_{h,J,\omega}^{(2)}}+\abs{{R}_{h,J,\omega}^{(3)}}\notag\\
  & \leq \big[C\norm{\nabla h}_{L^\infty}\norm{\Psi}_{X_{r,\tau,s}}+C(1+\tau)\norm{\Psi}_{H^r}^2\big]\norm{\Psi}_{X_{r,\tau,s}}\notag\\
  & \quad+\big[C\tau\norm{\nabla h}_{L^\infty}+C\tau^2(\norm{\Psi}_{H^r}+\norm{\Psi}_{X_{r,\tau,s}})\big]\norm{\Psi}_{Y_{r,\tau,s}}^2.
\end{align}
Combining \eqref{3.29} and \eqref{3.30}, we proved \eqref{3.21}.
\end{proof}
\begin{proof}[Proof of \eqref{3.22}]
It suffices to estimate $\inner{J\cdot\nabla h,\Lambda_m^{2r}e^{2\tau\Lambda_m^{1/s}}\omega}_{L^2(\mathbb{T}^3)}$, since the other term $\inner{\omega\cdot\nabla h,\Lambda_m^{2r}e^{2\tau\Lambda_m^{1/s}}J}_{L^2(\mathbb{T}^3)}$ can be estimated in similar way (replacing the position of $\omega$ and $J$).
First of all, we note that $\norm{J}_{L^\infty}\leq \norm{\nabla h}_{L^\infty}$,
\begin{align}\label{3.31}
  & \abs{\inner{\Lambda_m^r e^{\tau\Lambda_m^{1/s}}J\cdot\nabla h,\Lambda_m^r e^{\tau\Lambda_m^{1/s}}\omega}_{L^2(\mathbb{T}^3)}}+\abs{\inner{J\cdot\nabla\Lambda_m^r e^{\tau\Lambda_m^{1/s}} h,\Lambda_m^r e^{\tau \Lambda_m^{1/s}}\omega}_{L^2(\mathbb{T}^3)}} \notag\\
   &\leq C\norm{\nabla h}_{L^\infty} \norm{J}_{X_{r,\tau,s}}\norm{\Lambda_m^r e^{\tau\Lambda_m^{1/s}}\omega}_{L^2}.
\end{align}
Then like \eqref{3.15}, we substract $\inner{J\cdot\nabla h,\Lambda_m^{2r}e^{2\tau\Lambda_m^{1/s}}\omega}_{L^2(\mathbb{T}^3)}$ by
$$
\inner{\Lambda_m^r e^{\tau\Lambda_m^{1/s}}J\cdot\nabla h,\Lambda_m^r e^{\tau\Lambda_m^{1/s}}\omega}_{L^2(\mathbb{T}^3)}+\inner{J\cdot\nabla \Lambda_m^r e^{\tau\Lambda_m^{1/s}}h,\Lambda_m^r e^{\tau\Lambda_m^{1/s}}\omega}_{L^2(\mathbb{T}^3)}
$$
and we consider their differences
\begin{align}\label{3.32}
  &\inner{J\cdot\nabla h,\Lambda_m^{2r}e^{2\tau\Lambda_m^{1/s}}\omega}_{L^2(\mathbb{T}^3)}-\inner{\Lambda_m^r e^{\tau\Lambda_m^{1/s}}J\cdot\nabla h,\Lambda_m^r e^{\tau\Lambda_m^{1/s}}\omega} _{L^2(\mathbb{T}^3)}\notag\\
  &\quad-\inner{J\cdot\nabla \Lambda_m^r e^{\tau\Lambda_m^{1/s}}h,\Lambda_m^r e^{\tau\Lambda_m^{1/s}}\omega}_{L^2(\mathbb{T}^3)} \notag\\
  &=i(2\pi)^3\sum_{j+k+\ell=0}(\abs{\ell_m}^r-\abs{j_m}^r)(e^{\tau\abs{\ell_m}^{1/s}}-e^{\tau\abs{k_m}^{1/s}})(\hat{J}_j\cdot k)(\hat{h}_k\cdot\hat{\omega}_\ell)\abs{\ell_m}^r e^{\tau\abs{\ell_m}^{1/s}} \notag\\
  &\quad+i(2\pi)^3\sum_{j+k+\ell=0}(\abs{\ell_m}^r-\abs{k_m}^r-\abs{j_m}^r)e^{\tau\abs{k_m}^{1/s}}(\hat{J}_j\cdot k)(\hat{h}_k\cdot\hat{\omega}_\ell)\abs{\ell_m}^r e^{\tau\abs{\ell_m}^{1/s}} \notag\\
  &\quad+i(2\pi)^3\sum_{j+k+\ell=0}\abs{j_m}^r(e^{\tau\abs{\ell_m}^{1/s}}-e^{\tau\abs{j_m}^{1/s}})(\hat{J}_j\cdot k)(\hat{h}_k\cdot\hat{\omega}_\ell)\abs{\ell_m}^r e^{\tau\abs{\ell_m}^{1/s}} \notag\\
  &:=\mathcal{T}_{J,h,\omega}^{(1)}+\mathcal{T}_{J,h,\omega}^{(2)}+\mathcal{T}_{J,h,\omega}^{(3)}.
\end{align}
It rested to estimate the right hand side of \eqref{3.32}. Analogue to \eqref{3.15}, the three terms $\mathcal{T}_{J,h,\omega}^{(1)}$, $\mathcal{T}_{J,h,\omega}^{(2)}$ and $\mathcal{T}_{J,h,\omega}^{(3)}$ are estimated in the same way. Then we directly have %we appeal to the mean value theorem for $\abs{\ell_m}^r-\abs{j_m}^r$, and $\abs{e^x- 1}\leq\abs{x} e^{\abs{x}}$, for $x\in\mathbb{R}$, and the inequality \eqref{3.16},
%\begin{equation*}%\label{3.26}
%  \abs{(\abs{\ell_m}^r-\abs{j_m}^r)(e^{\tau\abs{\ell_m}^{1/s}}-e^{\tau\abs{k_m}^{1/s}}) }\leq C\tau \frac{\abs{j_m}\abs{k_m}^r+\abs{k_m}\abs{j_m}^r}{\abs{\ell_m}^{1-1/s}+\abs{k_m}^{1-1/s}} e^{\tau\abs{j_m}^{1/s}}e^{\tau\abs{k_m}^{1/s}}.
%\end{equation*}
%Combined with $e^x\leq 1+x e^x$, for all $x\geq0$, used for $e^{\tau(\abs{j_m}^{1/s}+\abs{k_m}^{1/s})}$, then we have
\begin{align}\label{3.33}
  \abs{\mathcal{T}_{J,h,\omega}^{(1)}} & \leq C\tau\norm{J}_{H^r}^2\norm{\Lambda_m^r e^{\tau\Lambda_m^{1/s}}\omega}_{L^2}\notag \\
  &\quad+C\tau^2 \norm{J}_{X_{r,\tau,s}}\norm{J}_{Y_{r,\tau,s}}\norm{\Lambda_m^{r+\frac{1}{2s}}e^{\tau\Lambda_m^{1/s}}\omega}_{L^2},
\end{align}
%where we implicitly used the Lemma \ref{Lemma 3.3} for $h$ in the above estimate.
%For the second term $\mathcal{T}_{J,h,\omega}^{(2)}$, by the mean value theorem we have
%\begin{equation*}%\label{3.28}
%  \abs{(\abs{\ell_m}^r-\abs{k_m}^r)-\abs{j_m}^r}\leq C\abs{j_m}(\abs{j_m}^{r-1}+\abs{k_m}^{r-1})+\abs{j_m}^r.
%\end{equation*}
%Using the inequality $e^x\leq e+x^2 e^x$, for all $x\geq 0$, for $e^{\tau\abs{k_m}^{1/s}}$, then we obtain
and
\begin{equation}\label{3.34}
  \abs{\mathcal{T}_{J,h,\omega}^{(2)}}\leq C\norm{J}_{H^r}^2 \norm{\Lambda_m^r e^{\tau\Lambda_m^{1/s}}\omega}_{L^2}+C\tau^2 \norm{J}_{H^r}\norm{J}_{Y_{r,\tau,s}}\norm{\Lambda_m^{r+\frac{1}{2s}} e^{\tau\Lambda_m^{1/s}}\omega}_{L^2},
\end{equation}
%where we used $\abs{k_m}^{\frac{1}{2s}}\leq \abs{j_m}^{\frac{1}{2s}}+\abs{\ell_m}^{\frac{1}{2s}}$ and the Lemma \ref{Lemma 3.3} implicitly. For the third term $\mathcal{T}_{J,h,\omega}^{(3)}$, we use the inequality $\abs{e^x-1}\leq \abs{x}e^{\abs{x}}$, for $x\in\mathbb{R}$, for $e^{\tau(\abs{\ell_m}^{1/s}-\abs{j_m}^{1/s})}$, and the inequality $e^x\leq 1+x e^x$, for $x\geq0$, for $e^{\tau\abs{j_m}^{1/s}}$, and the triangle inequality
%\begin{equation*}%\label{3.29}
%  \abs{j_m}^{\frac{1}{2s}}\leq \abs{k_m}^{\frac{1}{2s}}+\abs{\ell_m}^{\frac{1}{2s}}.
%\end{equation*}
and
\begin{align}\label{3.35}
  \abs{\mathcal{T}_{J,h,\omega}^{(3)}} &\leq C\tau\norm{J}_{H^r}^2 \norm{\Lambda_m^r e^{\tau\Lambda_m^{1/s}}\omega}_{L^2} \notag\\
  &\quad+ C\tau^2\norm{J}_{X_{r,\tau,s}}\norm{J}_{Y_{r,\tau,s}}
  \norm{\Lambda_m^{r+\frac{1}{2s}}e^{\tau\Lambda_m^{1/s}}\omega}_{L^2},
\end{align}
for an appropriate constant $C$.
%where we used Lemma \ref{Lemma 3.3} implicitly. 
Combining \eqref{3.33}, \eqref{3.34} and \eqref{3.35} with \eqref{3.31}, we have the estimate
\begin{align}\label{3.36}
  &\abs{\inner{J\cdot\nabla h,\Lambda_m^{2r}e^{2\tau\Lambda_m^{1/s}}\omega}_{L^2(\mathbb{T}^3)}}\notag\\
   & \leq C\norm{\nabla h}_{L^\infty} \norm{J}_{X_{r,\tau,s}}\norm{\Lambda_m^r e^{\tau\Lambda_m^{1/s}}\omega}_{L^2}\notag \\
  &\quad+ C\tau\norm{\Psi}_{H^r}^2\norm{\Lambda_m^r e^{\tau\Lambda_m^{1/s}}J}_{L^2}\notag \\
  &\quad+ C\tau^2(\norm{J}_{H^r}+\norm{J}_{X_{r,\tau,s}})\norm{J}_{Y_{r,\tau,s}}\norm{\Lambda_m^{r+\frac{1}{2s}}e^{\tau\Lambda_m^{1/s}}\omega}_{L^2}.
\end{align}
Symmetrically, we have
\begin{align}\label{3.37}
  &\abs{\inner{\omega\cdot\nabla h,\Lambda_m^{2r}e^{2\tau\Lambda_m^{1/s}}J}_{L^2(\mathbb{T}^3)}} \notag\\
  & \leq C\norm{\nabla h}_{L^\infty} \norm{\omega}_{X_{r,\tau,s}}\norm{\Lambda_m^r e^{\tau\Lambda_m^{1/s}}J}_{L^2}\notag \\
  &\quad+C\norm{\nabla u}_{L^\infty} \norm{J}_{X_{r,\tau,s}}^2+ C\tau\norm{\omega}_{H^r}\norm{J}_{H^r}\norm{\Lambda_m^r e^{\tau\Lambda_m^{1/s}}\omega}_{L^2}\notag \\
  &\quad+ C\tau^2(\norm{\Psi}_{H^r}+\norm{\Psi}_{X_{r,\tau,s}})\norm{\Psi}_{Y_{r,\tau,s}}\norm{\Lambda_m^{r+\frac{1}{2s}}e^{\tau\Lambda_m^{1/s}}J}_{L^2}.
\end{align}
Combining \eqref{3.36} with \eqref{3.37}, we have
\begin{align*}
  & \abs{\inner{J\cdot\nabla h,\Lambda_m^{2r}e^{2\tau\Lambda_m^{1/s}}\omega}_{L^2(\mathbb{T}^3)}}+\abs{\inner{\omega\cdot\nabla h,\Lambda_m^{2r}e^{2\tau\Lambda_m^{1/s}}J}_{L^2(\mathbb{T}^3)}} \\
  &\leq C(\norm{\nabla u}_{L^\infty}+\norm{\nabla h}_{L^\infty})\norm{\Psi}_{X_{r,\tau,s}}^2+C\tau\norm{\Psi}_{H^r}^2\norm{\Psi}_{X_{r,\tau,s}}\\
  &\qquad+C\tau^2 (\norm{\Psi}_{H^r}+\norm{\Psi}_{X_{r,\tau,s}})\norm{\Psi}_{Y_{r,\tau,s}}^2
\end{align*}
Then \eqref{3.22} is proved.
\end{proof}

\section{Proof of Theorem \ref{Theorem 2.1}}\label{Section 4}
In this Section, we will give the proof of the main theorem. Here we present only a priori estimate, since the rigorous construction of the solution follows from the standard Galerkin approximation.
\begin{proof}[Proof of Theorem \ref{Theorem 2.1}]
For simplicity of presentation we suppress the time dependence of $\tau, u, h, \omega$ and $J$ on $t$.
As usual, let $m\in\{1,2,3\}$, let us take the $L^2$ inner product of the first equation of \eqref{1.2} with $\Lambda_m^{2r} e^{2\tau\Lambda_m^{1/s}}\omega$, and the second equation of \eqref{1.2} with $\Lambda_m^{2r}e^{2\tau\Lambda_m^{1/s}}J$ respectively,
\begin{align}\label{4.1}
  &\inner{\partial_t \omega,\Lambda_m^{2r}e^{2\tau\Lambda_m^{1/s}}\omega}_{L^2(\mathbb{T}^3)} +\inner{u\cdot\nabla\omega-\omega\cdot\nabla u,\Lambda_m^{2r}e^{2\tau\Lambda_m^{1/s}}\omega}_{L^2(\mathbb{T}^3)}\notag \\
  &-\inner{h\cdot\nabla J,\Lambda_m^{2r}e^{2\tau\Lambda_m^{1/s}}\omega}_{L^2(\mathbb{T}^3)}+\inner{J\cdot\nabla h,\Lambda_m^{2r}e^{2\tau\Lambda_m^{1/s}}\omega}_{L^2(\mathbb{T}^3)}=0,
\end{align}
and
\begin{align}\label{4.2}
  &\inner{\partial_t J,\Lambda_m^{2r}e^{2\tau\Lambda_m^{1/s}}J}_{L^2(\mathbb{T}^3)} +\inner{u\cdot\nabla J+J\cdot\nabla u,\Lambda_m^{2r}e^{2\tau\Lambda_m^{1/s}}J}_{L^2(\mathbb{T}^3)}\notag \\
  &-\inner{h\cdot\nabla \omega,\Lambda_m^{2r}e^{2\tau\Lambda_m^{1/s}}J}_{L^2(\mathbb{T}^3)}-\inner{\omega\cdot\nabla h,\Lambda_m^{2r}e^{2\tau\Lambda_m^{1/s}}J}_{L^2(\mathbb{T}^3)}=0.
\end{align}
Adding \eqref{4.1} and \eqref{4.2} together, we have
\begin{equation}\label{4.3}
\frac{1}{2}\frac{d}{dt}\norm{\Lambda_m^r e^{\tau\Lambda_m^{1/s}}\Psi}_{L^2}^2 =\dot{\tau}(t)\norm{\Lambda_m^{r+\frac{1}{2s}} e^{\tau\Lambda_m^{1/s}}\Psi}_{L^2}^2+K_1+K_2+K_3,
\end{equation}
where $K_1, K_2$ and $K_3$ are as follows,
\begin{align*}
  K_1 &=-\inner{u\cdot\nabla\omega-\omega\cdot\nabla u,\Lambda_m^{2r}e^{2\tau\Lambda_m^{1/s}}\omega}_{L^2(\mathbb{T}^3)}-\inner{u\cdot\nabla J+J\cdot\nabla u,\Lambda_m^{2r}e^{2\tau\Lambda_m^{1/s}}J}_{L^2(\mathbb{T}^3)}\notag \\
  K_2 &=\inner{h\cdot\nabla J,\Lambda_m^{2r}e^{2\tau\Lambda_m^{1/s}}\omega}_{L^2(\mathbb{T}^3)}+\inner{h\cdot\nabla \omega,\Lambda_m^{2r}e^{2\tau\Lambda_m^{1/s}}J}_{L^2(\mathbb{T}^3)}\notag \\
  K_3 &=-\inner{J\cdot\nabla h,\Lambda_m^{2r}e^{2\tau\Lambda_m^{1/s}}\omega}_{L^2(\mathbb{T}^3)}+\inner{\omega\cdot\nabla h,\Lambda_m^{2r}e^{2\tau\Lambda_m^{1/s}}J}_{L^2(\mathbb{T}^3)}.
\end{align*}
By the \eqref{3.1} in Lemma \ref{Lemma 3.4} and \eqref{3.2} in Lemma \ref{lemma 3.5}, we have
\begin{align*}%\label{4.4}
  \abs{K_1} &\leq C(\tau\norm{\nabla u}_{L^\infty}+\tau^2\norm{\Psi}_{H^r}+\tau^2\norm{\Psi}_{X_{r,\tau,s}})\norm{\Psi}_{Y_{r,\tau,s}}^2\notag\\
                  &\qquad+C\big(\norm{\nabla u}_{L^\infty}\norm{\Psi}_{X_{r,\tau,s}}+(1+\tau)
                  \norm{\Psi}_{H^r}^2\big)\norm{\Psi}_{X_{r,\tau,s}}\,.
\end{align*}
By \eqref{3.21} in the Lemma \ref{Lemma 3.6}, we have
\begin{align*}%\label{4.5}
  \abs{K_2} &\leq C(\tau\norm{\nabla h}_{L^\infty}+\tau^2\norm{\Psi}_{H^r}+\tau^2\norm{\Psi}_{X_{r,\tau,s}})\norm{\Psi}_{Y_{r,\tau,s}}^2\notag\\
                  &\qquad+C\big(\norm{\nabla h}_{L^\infty}\norm{\Psi}_{X_{r,\tau,s}}+(1+\tau)
                  \norm{\Psi}_{H^r}^2\big)\norm{\Psi}_{X_{r,\tau,s}}\,.
\end{align*}
By \eqref{3.22} in the Lemma \ref{Lemma 3.6}, we have
\begin{align*}%\label{4.6}
  \abs{K_3} & \leq C\norm{\nabla h}_{L^\infty}\norm{\Psi}_{X_{r,\tau,s}}^2+C\tau\norm{\Psi}_{H^r}^2\norm{\Psi}_{X_{r,\tau,s}}\\
  &\qquad+C\tau^2 (\norm{\Psi}_{H^r}+\norm{\Psi}_{X_{r,\tau,s}})\norm{\Psi}_{Y_{r,\tau,s}}^2\,.
\end{align*}
Substituting $K_1, K_2, K_3$ into \eqref{4.3}, we have
\begin{align*}%\label{4.7}
  \frac{1}{2}\frac{d}{dt} &\norm{\Lambda_m^r e^{\tau\Lambda_m^{1/s}}\Psi}_{L^2}^2 \leq \dot{\tau}(t)\norm{\Lambda_m^{r+\frac{1}{2s}} e^{\tau\Lambda_m^{1/s}}\Psi}_{L^2}^2\notag\\
  &+C\big[(\norm{\nabla u}_{L^\infty}+\norm{\nabla h}_{L^\infty})\norm{\Psi}_{X_{r,\tau,s}}+(1+\tau)\norm{\Psi}_{H^r}^2\big]\norm{\Psi}_{X_{r,\tau,s}}  \notag\\
   &+C\big[\tau(\norm{\nabla u}_{L^\infty}+\norm{\nabla h}_{L^\infty})\norm{\Psi}_{X_{r,\tau,s}}+\tau^2(\norm{\Psi}_{H^r}+
   \norm{\Psi}_{X_{r,\tau,s}})\big]\norm{\Psi}_{Y_{r,\tau,s}}^2\,.
\end{align*}
Taking summation from $m=1$ to $m=3$, we have
\begin{align*}%\label{4.8}
  \frac{1}{2}&\frac{d}{dt} \norm{\Psi}_{X_{r,\tau,s}}^2 \leq C\big[(\norm{\nabla u}_{L^\infty}+\norm{\nabla h}_{L^\infty})\norm{\Psi}_{X_{r,\tau,s}}+(1+\tau)\norm{\Psi}_{H^r}^2\big]\norm{\Psi}_{X_{r,\tau,s}}  \notag\\
   &+\big[\dot{\tau}+C\tau(\norm{\nabla u}_{L^\infty}+\norm{\nabla h}_{L^\infty})+C\tau^2(\norm{\Psi}_{H^r}+
   \norm{\Psi}_{X_{r,\tau,s}})\big]\norm{\Psi}_{Y_{r,\tau,s}}^2\,.
\end{align*}
If $\tau(t)$ is a decreasing function of $t$ such that
\begin{equation}\label{4.9}
  \dot{\tau}+C\tau(\norm{\nabla u}_{L^\infty}+\norm{\nabla h}_{L^\infty})\norm{\Psi}_{X_{r,\tau,s}}+C\tau^2(\norm{\Psi}_{H^r}+\norm{\Psi}_{X_{r,\tau,s}})\leq 0
\end{equation}
Then we have
\begin{equation}\label{4.10}
  \frac{d}{dt} \norm{\Psi}_{X_{r,\tau(t)}} \leq C(\norm{\nabla u}_{L^\infty}+\norm{\nabla h}_{L^\infty})\norm{\Psi}_{X_{r,\tau,s}}+C(1+\tau(0))\norm{\Psi}_{H^r}^2.
\end{equation}
By standard $H^r$-energy estimate one can obtain that there exists a constant $\tilde{C}>0$ depending on $r$ such that
\begin{equation}\label{4.10+}
  \norm{\Psi(\cdot,t)}_{H^r}\leq \norm{\Psi_0}_{H^r}\exp\big(\int_{0}^{t}\tilde{C}(\norm{\nabla u(\cdot,\sigma)}_{L^\infty}+\norm{\nabla h(\cdot,\sigma)}_{L^\infty})d \sigma\big),
\end{equation}
for $0<t<T$.
We now let the constant $C$ large enough such that \eqref{4.10+} holds. By Grownwall's inequality in \eqref{4.10}, we have
\begin{equation*}%\label{4.11}
\begin{split}
  \norm{\Psi(\cdot,t)}_{X_{r,\tau(t),s}} &\leq G(t)\bigg[\norm{\Psi_0}_{X_{r,\tau(0),s}}+C(1+\tau(0))\int_0^t \norm{\Psi(\cdot,\sigma)}_{H^r}^2 G^{-1}(\sigma)d\sigma \bigg] \\
      &:=M(t)\\
      &\leq G(t)\big(\norm{\Psi_0}_{X_{r,\tau_0,s}}+C_{\tau_0}\norm{\Psi_0}_{H^r}^2 t \big),
\end{split}
\end{equation*}
where we denote
$$
G(t)=\exp\bigg(\int_{0}^{t}C(\norm{\nabla u(\cdot,\sigma)}_{L^\infty}+\norm{\nabla h(\cdot,\sigma)}_{L^\infty}) d\sigma\bigg)
$$
and $C_{\tau_0}=C(1+\tau(0))$.
A sufficient condition \eqref{4.9} to hold is that $\tau$ satisfies
\begin{equation*}%\label{4.12}
  \dot{\tau}(t)+C\tau(\norm{\nabla u}_{L^\infty}+\norm{\nabla h}_{L^\infty})+C\tau^2 \norm{\Psi}_{H^r}+C\tau^2 M(t)= 0,
\end{equation*}
for all $0<t<T$. It suffices to set
\begin{equation*}%\label{4.13}
  \tau(t) =G(t)^{-1}\bigg[\tau(0)^{-1}+C\int_{0}^{t}\big(\norm{\Psi(\cdot,\sigma)}_{H^r}+M(\sigma) \big)G(\sigma)^{-1}d\sigma \bigg]^{-1}.
\end{equation*}
In particular, since $\norm{\Psi(\cdot,t)}_{H^r}^2\leq \norm{\Psi_0}_{H^r}^2 G(t)$, we obtain
\begin{equation*}%\label{4.14}
\begin{split}
  \tau(t) &\geq G(t)^{-1}\bigg\{\tau_0^{-1}+C\int_{0}^{t}\big[\norm{\Psi_0}_{H^r}+(\norm{\Psi_0}_{X_{r,\tau_0,s}}+C_{\tau_0}\norm{\Psi_0}^2_{H^r}\sigma )\big]d\sigma\bigg\}^{-1} \\
     &\geq G(t)^{-1} \big(\tau_0^{-1}+C_0 t+\frac{C_1}{2}t^2\big)^{-1}
\end{split}
\end{equation*}
where $C_0=C(\norm{\Psi_0}_{H^r}+\norm{\Psi_0}_{X_{r,\tau_0,s}})$ and the constant $C_1=\frac{CC_{\tau_0}\norm{\Psi_0}^2_{H^r}}{2}$.
\end{proof}

\bigskip
\noindent{\bf Acknowledgements.}
The research of the second author is supported partially by
``The Fundamental Research Funds for Central Universities of China".


\begin{thebibliography}{99}
\bibitem{BK}
\newblock Beal, J.T; Kato, T; Majda, A.
\newblock Remarks on the breakdown of smooth solutions for the 3-D Euler equations.
\newblock {\em Communications in Mathematical Physics}, 1984, 94(1): 61-65.

\bibitem{CCM}
\newblock Cannone M; Chen Q; Miao C.
\newblock A losing estimate for the ideal MHD equations with application to blow-up criterion.
\newblock {\em SIAM Journal on Mathematical Analysis}, 2007, 38(6): 1847-1859.

\bibitem{CKS}
\newblock Caflisch, R.E; Klapper, I; Steele, G.
\newblock Remarks on singularities, dimension and energy dissipation for ideal hydrodynamics and MHD.
\newblock {\em Comm. Math. Phys.} 184(1997), 443-455.


\bibitem{FT}
\newblock Foias, C; Temam, R.
\newblock Gevrey class regularity for the solutions of the Navier-Stokes equations.
\newblock {\em J. Funct. Anal}, 87(1989), 359-369.

\bibitem{KL}
\newblock Kato, T; Lai, C.Y.
\newblock Nonlinear evolution equations and the Euler flow.
\newblock {\em J. Funct. Anal.}. 56(1984), 15-28.

\bibitem{KLT}
\newblock Kalantarov, V.K; Levant, B; Titi, E.S.
\newblock Gevrey regularity for the global attractor of the 3D Navier-Stokes-Voight equations.
\newblock {\em J. Nonlinear Sci.}, 19(2009), 133-152.

\bibitem{KV1}
\newblock Kukavica, I; Vicol, V.
\newblock On the radius of analyticity of solutions to the three-dimensional Euler equations.
\newblock {\em Proc. Amer. Math. Soc}. 137(2009), 669-677.

\bibitem{LL}
\newblock Laudau, L.D; Lifshitz, E.M.
\newblock Electrondynamics of Continuous Media.
\newblock 2nd ed. Pergamon, New York, 1984.

\bibitem{LLP}
\newblock Ling-Bing, He; Li, Xu; Pin, Yu.
\newblock On global dynamics of three dimensional Magnetohydrodynamics: Nonlinear Stability of Alfv\'en waves.
\newblock arxiv:1603.08205v1. 2016.

\bibitem{LO}
\newblock Levermore, C.D; Oliver, M.
\newblock Analyticity of solutions for a generalized Euler equation.
\newblock {\em J. Differential Equations}, 133(1997), 321-339.

\bibitem{MB}
\newblock Majda, A.J; Bertozzi, A.L.
\newblock Vorticity and incompressible flow.
\newblock Cambridge University Press, 2002.

\bibitem{SK}
\newblock Sangjeong, K.
\newblock Gevrey class regularity of the magnetohydrodynamics equations.
\newblock {\em ANZIAM J.}, 43(2002), no. 3, 397-408.

\bibitem{ST}
\newblock Sermange, M; Temam, R.
\newblock Some mathematical questions related to the MHD equations.
\newblock {\em Comm. Pure Appl. Math.}. 36(1983), 635-664.

\bibitem{SP}
\newblock Secchi P.
\newblock On the equations of ideal incompressible magneto-hydrodynamics.
\newblock {\em Rendiconti del Seminario Matematico della Universite di Padova}, 1993, 90: 103-119.

\bibitem{T}
\newblock Temam, R.
\newblock On the Euler equations of incompressible perfect fluids.
\newblock {\em J. Functional Analysis}, 20(1975), 32-43.


\bibitem{WL}
\newblock Wang, Y. Z; Li, P.
\newblock Global existence of three dimensional incompressible MHD flows.
\newblock {\em Mathematical Methods in the Applied Sciences}, 2016.

\bibitem{W}
\newblock Weng, S.
\newblock On analyticity and temporal decay rates of solutions to the viscous resistive Hall-MHD system.
\newblock {\em Journal of Differential Equations}, 2016, 260(8): 6504-6524.

\bibitem{Wu}
\newblock J.Wu
\newblock Bounds and new approaches for the 3D MHD equations.
\newblock {\em J. Nonlinear Sci.}, 12 (2002), pp. 395-413.

\bibitem{YL}
\newblock Yongjiang Yu; Kaitai, Li.
\newblock Existence of solutions for the MHD-Leray-alpha equations and their relations to the MHD equations.
\newblock {\em J. Math. Anal. Appl.} 329(2007), 298-326.

\bibitem{YZ}
\newblock Yuan, C; Zhen, L.
\newblock Global well-posedness of the Incompressible Magnetohydrodynamics.
\newblock arxiv: 1605.00439v1, 2016.

\bibitem{ZX}
\newblock Z-F. Zhang; X-F. Liu.
\newblock On the blow-up criterion of smooth solutions to the 3D ideal MHD equations.
\newblock {\em Acta Math. Appl. Sin. Engl. Ser., E}, 20 (2004), 695-700.

\bibitem{ZL}
\newblock Zhao, C; Li, B.
\newblock Analyticity of the global attractor for the 3D regularized MHD equations.
\newblock {\em Electronic Journal of Differential Equations}, 2016.

\end{thebibliography}
\end{document}